\documentclass[reqno]{amsart}
\usepackage{amssymb}
\usepackage{amsmath}
\usepackage[mathscr]{euscript}
\usepackage{times}

\usepackage{color}

\makeatletter
\@addtoreset{equation}{section}
\makeatother

\newtheorem{theorem}{Theorem}[section]
\newtheorem{lemma}[theorem]{Lemma}
\newtheorem{proposition}[theorem]{Proposition}
\newtheorem{corollary}[theorem]{Corollary}

\newtheorem{definition}[theorem]{Definition}
\newtheorem{remark}[theorem]{Remark}

\newcommand{\bb}[1]{{\mathbb #1}}

\newcounter{as}[section]

\renewcommand{\a}{\alpha}

\newcommand{\e}{\varepsilon}
\newcommand{\de}{\delta}

\newcommand{\ga}{\gamma}

\newcommand{\la}{\lambda}

\newcommand{\De}{\Delta}

\newcommand{\lam}{\lambda}

\newcommand{\lan}{\langle}
\newcommand{\ran}{\rangle}

\def\R{{\mathbb{R}}}
\def\E{{\mathcal{E}}}
\def\N{{\mathbb{N}}}
\def\Z{{\mathbb{Z}}}

\def\T{{\mathbb{T}}}
\def\M{{\mathcal{M}}}

\begin{document}

\title{Hydrostatics and dynamical large deviations
for a reaction-diffusion model}
\author{C. Landim, K. Tsunoda}
\address{\noindent IMPA, Estrada Dona Castorina 110, CEP 22460 Rio de
  Janeiro, Brasil and CNRS UMR 6085, Universit\'e de Rouen, Avenue de
  l'Universit\'e, BP.12, Technop\^ole du Madril\-let, F76801
  Saint-\'Etienne-du-Rouvray, France. 
  \newline e-mail: \rm \texttt{landim@impa.br} }
\address{\noindent Institute of Mathematics for Industry, Kyushu University,
744, Motooka, Nishi-ku, Fukuoka, 819-0395, Japan. 
\newline e-mail:  \rm \texttt{k-tsunoda@imi.kyushu-u.ac.jp}}
\footnote{Abbreviated title $($running head$)$:
LDP for a reaction-diffusion model.}

\begin{abstract}
  We consider the superposition of a symmetric simple exclusion
  dynamics, speeded-up in time, with a spin-flip dynamics in a
  one-dimensional interval with periodic boundary conditions.  We
  prove the hydrostatics and the dynamical large deviation principle.
\end{abstract}

\maketitle

\section{Introduction}

In recent years, the large deviations of interacting particle systems
have attracted much attention as an important step in the foundation
of a thermodynamic theory of nonequilibrium stationary states
\cite{dls, bdgjl1, bd, bdgjl3}. Notwithstanding the absence of
explicit expressions for the stationary states, large deviations
principles for the empirical measure under the stationary state have
been derived from a dynamical large deviations principle \cite{bdgjl2,
  flm, bl}, extending to an infinite-dimensional setting \cite{bg, f}
for the Freidlin and Wentzell approach \cite{fw1}.

We consider in this article interacting particle systems in which a
symmetric simple exclusion dynamics, speeded-up diffusively, is
superposed to a non-conservative Glauber dynamics. De Masi, Ferrari
and Lebowitz \cite{dfl} proved that the macroscopic evolution of the
empirical measure is described by the solutions of
the reaction-diffusion equation
\begin{equation}
\label{i01}
\partial_t\rho \;=\; (1/2) \De \rho + B(\rho) - D(\rho) \;.
\end{equation}
where $\De$ is the Laplacian and $F=B-D$ is a reaction term determined
by the stochastic dynamics. They also proved that the equilibrium
fluctuations evolve as generalized Ornstein-Uhlenbeck processes.

A large deviation principle for the empirical measure has been
obtained in \cite{jlv} in the case where the initial distribution is a
local equilibrium. The lower bound of the large deviations principle
was achieved only for smooth trajectories. More recently, \cite{bl}
extended the large deviations principle to a one-dimensional dynamics
in contact with reservoirs and proved the lower bound for general
trajectories in the case where the birth and the death rates,
$B(\rho)$ and $D(\rho)$, respectively, are monotone, concave
functions.

In this article, we first present a law of large numbers for the
empirical measure under the stationary state \cite{ELS, KLO}. More
precisely, denote by $\mu_N$ the stationary state on a one-dimensional
torus with $N$ points of the superposition of a Glauber dynamics with
a symmetric simple exclusion dynamics speeded-up by $N^2$. This
probability measure is not known explicitly and it exhibits long range
correlations \cite{bj}. Let $V_\epsilon$ denote an
$\epsilon$-neighborhood of the set of solutions of the elliptic
equation
\begin{equation}
\label{i02}
(1/2)\De\rho + F(\rho) \;=\; 0 \;.
\end{equation}
Theorem \ref{hsl} asserts that for any $\epsilon>0$,
$\mu_N(V^c_\epsilon)$ vanishes as $N\to\infty$. In contrast with
previous results, equation \eqref{i02} may not have a unique solution
so that equation \eqref{i01} may not have a global attractor, what
prevents the use of the techniques developed in \cite{flm, mo}.
This result solves partially a conjecture raised in Subsection 4.2 of \cite{bl2}.

The main results of this article concern the large deviations of the
Glauber-Kawasaki dynamics. We first prove a full large deviations
principle for the empirical measure under the sole assumption that $B$
and $D$ are concave functions. These assumptions encompass the case in
which the potential $F(\rho) = B(\rho) - D(\rho)$ presents two or more
wells, and open the way to the investigation of the metastable
behavior of this dynamics. Previous results in this directions include
\cite{dpv, dppv, blm1}.

We also prove that the large deviations rate function is
lower semicontinuous and has compact level sets. These properties play
a fundamental role in the proof of the static large deviation
principle for the empirical measure under the stationary state $\mu_N$
\cite{bg, f}.

The main difficulty in the proof of the lower bound of the large
deviation principle comes from the presence of exponential terms in
the rate function, denoted in this introduction by $I$. In contrast
with conservative dynamics, for a trajectory $u(t,x)$, $I(u)$ is not
expressed as a weighted $H_{-1}$ norm of $\partial_t u - (1/2) \De u -
F(u)$. This forces the development of new tools to prove that smooth
trajectories are $I$-dense.

Both the large deviations of the empirical measure under the
stationary state and the metastable behavior of the dynamics in the
case where the potential admits more than one well are investigated in
\cite{flt} based on the results presented in this article.  \medskip

\noindent{\bf Comments on the proof}.  The proof of the law of large
numbers for the empirical measure under the stationary state $\mu_N$
borrows ideas from \cite{flm, mo}. On the one hand, by \cite{dfl}, the
evolution of the empirical measure is described by the solutions of
the reaction-diffusion equation \eqref{i01}.  On the other hand, by
\cite{cm}, for any density profile $\gamma$, the solution $\rho_{t}$
of \eqref{i01} with initial condition $\gamma$ converges to some
solution of the semilinear elliptic equation \eqref{i02}. Assembling
these two facts, we show in the proof of Theorem \ref{hsl} that the
empirical measure eventually reaches a neighborhood of the set of all
solutions of the semilinear elliptic equation \eqref{i02}.

The proof that the rate function $I$ is lower semicontinuous and has
compact level set is divided in two steps. Denote by $Q(\pi)$ the
energy of a trajectory $\pi$, defined in \eqref{i03}. Following
\cite{qrv}, we first show in Proposition \ref{energy} that the energy
of a trajectory $\pi$ is bounded by the sum of its rate function with
a constant: $Q(\pi) \le C_0 (I(\pi)+1)$. It is not difficult to show
that a sequence in the set $\{\pi : Q(\pi) \le a\}$, $a>0$, which
converges weakly also converges in $L^1$. The lower semicontinuity of
the rate function $I$ follows from these two facts. Let $\pi_n$ be a
sequence which converges weakly to $\pi$. We may, of course, assume
that the sequence $I(\pi_n)$ is bounded. In this case, by the two
results presented above, $\pi_n$ converges to $\pi$ in $L^1$. As the
rate function $I(\cdot)$, defined in \eqref{i04}, is given by
$\sup_{G} J_G (\cdot)$, where the supremum is carried over smooth
functions, and since for each such function $J_G$ is continuous for
the $L^1$ topology, $J_G(\pi) = \lim_n J_G(\pi_n) \le \liminf_n
I(\pi_n)$. To conclude the proof of the lower semicontinuity of $I$,
it remains to maximize over $G$. The proof that the level sets are
compact is similar.

Note that the previous argument does not require a bound of the
$H_{-1}$ norm of $\partial_t \pi$ in terms of $I(\pi)$ and
$Q(\pi)$. Actually, such a bound does not hold in the present
context. For example, let $\rho$ represent the solution of the
hydrodynamic equation \eqref{i01} starting from some initial condition
$\gamma$. Due to the reaction term, the $H_{-1}$ norm of $\partial_t
\rho$ might be infinite, while $I(\rho)=0$ and $Q(\rho)<\infty$.
The fact that a bound on the $H_{-1}$ norm of $\partial_t \pi$ is not
used, may simplify the earlier proofs of the regularity of the rate
function in the case of conservative dynamics \cite{blm2, flm}.

The main difficulty in the proof of the lower bound lies in the
$I$-density of smooth trajectories: each trajectory $\pi$ with finite
rate function should be approachable by a sequence of smooth
trajectories $\pi_n$ such that $I(\pi_n)$ converges to $I(\pi)$. We
use in this step the hydrodynamic equation and several convolutions
with mollifiers to smooth the paths. The concavity of $B$ and $D$ are
used in this step and only in this one.
We emphasize that we can not use Theorem 2.4 in \cite{jlv}
in our setting due to the large deviations
which come from initial configurations.
Therefore we need to prove the $I$-density, Theorem \ref{dense}.
It is possible that the theory
of Orlicz spaces may allow to weaken these assumptions.  Similar
difficulties appeared in the investigation of the large deviations of
a random walk driven by an exclusion process and of the exclusion
process with a slow bond \cite{ajv, fn}.

This article is organized as follows.  In Section 2, we introduce a
reaction-diffusion model and state the main results. In Section 3 we
prove the law of large numbers for the empirical measure under the
stationary state. In Section 4, we present the main properties of the
rate function $I$. In Section 5, we prove that the smooth trajectories
are $I$-dense and we prove Theorem \ref{mt1}, the main result of the
article. In Section 6, we recall some results on the solution of the
hydrodynamic equation \eqref{i01}.

\section{Notation and Results}

Throughout this article, we use the following notation.  $\N_{0}$
stands for the set $\{0,1,\cdots\}$.  For a function $f: X\to \bb R$,
defined on some space $X$, let $\|f\|_\infty = \sup_{x\in X}{|f(x)|}$.
We will use $C_{0}>0$ and $C>0$ as a notation for a generic positive
constant which may change from line to line.

\subsection{Reaction-diffusion model}\label{model}
We fix some notation and define the model.  Let $\T_N$ be the
one-dimensional discrete torus $\Z/N\Z=\{0,1, \cdots, N-1\}$.  The
state space of our process is given by $X_N=\{0,1\}^{\T_N}$.  Let
$\eta$ denote a configuration in $X_N$, $x$ a site in $\T_N$,
$\eta(x)=1$ if there is a particle at site $x$, otherwise $\eta(x)=0$.

We consider in the set $\T_N$ the superposition of the symmetric
simple exclusion process (Kawasaki) with a spin-flip dynamics
(Glauber).  This model was introduced by De Masi, Ferrari and Lebowitz
in \cite{dfl} to derive a reaction-diffusion equation from a
microscopic dynamics.  More precisely, the stochastic dynamics is a
Markov process on $X_N$ whose generator $\mathcal{L}_N$ acts on
functions $f:X_N\to\R$ as
\begin{align*}
\mathcal{L}_Nf \;=\; \frac{N^2}{2} \mathcal{L}_Kf + \mathcal{L}_Gf\;,
\end{align*}
where $\mathcal{L}_K$ is the generator of a symmetric simple exclusion
process (Kawasaki dynamics),
\begin{align*}
(\mathcal{L}_Kf)(\eta) \;=\; \sum_{x\in\T_N} [f(\eta^{x,x+1}) - f(\eta)]\;,
\end{align*}
and where $\mathcal{L}_G$ is the generator of a spin flip dynamics
(Glauber dynamics),
\begin{align*}
(\mathcal{L}_Gf)(\eta) \;=\; \sum_{x\in\T_N} c(x,\eta)[f(\eta^x) - f(\eta)]\;.
\end{align*}
In these formulas, $\eta^{x,x+1}$ (resp. $\eta^x$) represents the
configuration obtained from $\eta$ by exchanging (resp. flipping) the
occupation variables $\eta(x)$, $\eta(x+1)$ (resp. $\eta(x)$):
\begin{align*}
\eta^x(z) \;=\; \begin{cases}
     \eta(z)  & \text{if $z\neq x$}\;, \\
     1-\eta(z)  & \text{if $z=x$}\;, 
\end{cases}
\quad
\eta^{x,y}(z) \;=\; \begin{cases}
     \eta(y)  & \text{if $z=x$}\;, \\
     \eta(x)  & \text{if $z=y$}\;, \\
     \eta(z)    & \text{otherwise}\;.
\end{cases}
\end{align*}
Moreover, $c(x,\eta)=c(\eta(x-M),\cdots, \eta(x+M))$, for some $M\ge
1$ and some strictly positive cylinder function $c(\eta)$, that is, a
function which depends only on a finite number of variables $\eta(y)$.
Note that the exclusion dynamics has been speeded-up by a factor
$N^{2}$, and that the Markov process generated by $\mathcal{L}_{N}$ is
irreducible because $c(\eta)$ is a strictly positive function.

\subsection{Hydrodynamic limit}
\label{hdl}
We briefly discuss in this subsection the limiting behavior of the
empirical measure.

Denote by $\T$ the one-dimensional continuous torus $\T=\R/\Z=[0,1)$.
Let $\mathcal{M}_+=\mathcal{M}_+(\T)$ be the space of nonnegative
measures on $\T$, whose total mass bounded by $1$, endowed with the
weak topology.  For a measure $\pi$ in $\mathcal{M}_+$ and a
continuous function $G:\T\to\R$, denote by $\lan \pi, G \ran$ the
integral of $G$ with respect to $\pi$:
\begin{equation*}
\lan \pi, G \ran \;=\; \int_\T G(u) \pi(du) \;.
\end{equation*}
The space $\M_+$ is metrizable. Indeed, if $f_{2k}(u)= \cos(\pi k u)$
and $f_{2k+1}(u)= \sin(\pi ku)$, $k\in\N_{0}$, one can define the
distance $d$ on $\M_{+}$ as
\begin{equation*}
d(\pi_{1},\pi_{2}) \;:=\; \sum_{k=0}^{\infty}\dfrac{1}{2^{k}}
|\lan\pi_{1},f_{k}\ran - \lan\pi_{2}, f_{k}\ran| \; .
\end{equation*}

Denote by $C^m(\T)$, $m$ in $\N_{0}\cup\{\infty\}$, the set of all
real functions on $\T$ which are $m$ times differentiable and whose
$m$-th derivative is continuous.  Given a function $G$ in $C^2(\T)$,
we shall denote by $\nabla G$ and $\De G$ the first and second
derivative of $G$, respectively.

Let $\{\eta_t^N:N\ge1\}$ be the continuous-time Markov process on
$X_{N}$ whose generator is given by $\mathcal{L}_{N}$.  Let $\pi^N:
X_N\to\M_+$ be the function which associates to a configuration $\eta$
the positive measure obtained by assigning mass $N^{-1}$ to each
particle of $\eta$,
\begin{align*}
\pi^N(\eta)\;=\;\frac{1}{N}\sum_{x\in \T_N}\eta(x)\de_{x/N}\;,
\end{align*}
where $\de_u$ stands for the Dirac measure which has a point mass at
$u\in\T$.  Denote by $\pi_{t}^{N}$ the empirical measure process
$\pi^{N}(\eta_{t}^{N})$.

Fix arbitrarily $T>0$.  For a topological space $X$ and an interval
$I=[0,T]$ or $[0,\infty)$, denote by $C(I,X)$ the set of all
continuous trajectories from $I$ to $X$ endowed with the uniform
topology.  Let $D(I,X)$ be the space of all right-continuous
trajectories from $I$ to $X$ with left-limits, endowed with the
Skorokhod topology.  For a probability measure $\nu$ in $X_N$, denote
by $\mathbb{P}_{\nu}^{N}$ the measure on $D([0,T], X_N)$ induced by
the process $\eta_{t}^{N}$ starting from $\nu$.

Let $\nu_\rho=\nu^N_{\rho}$, $0\le\rho\le 1$, be the Bernoulli product
measure with the density $\rho$.  Define the continuous functions
$B,D:[0,1]\to\R$ by
\begin{align*}
B(\rho)\;=\;\int [1-\eta(0)]\, c(\eta)\, d\nu_\rho\;, \quad
D(\rho)\;=\;\int \eta(0)\, c(\eta)\, d\nu_\rho \;.
\end{align*}
Since $B(1)=0$, $D(0)=0$ and $B,D$ are polynomials in $\rho$, 
\begin{equation}
\label{02}
B(\rho)\;=\; (1-\rho) \, \tilde B(\rho)\;, \quad
D(\rho)\;=\; \rho\, \tilde D(\rho)\;, 
\end{equation}
where $\tilde B(\rho)$, $\tilde D(\rho)$ are polynomials.

The next result was proved by De Masi, Ferrari and Lebowitz in
\cite{dfl} for the first time. We refer to \cite{dfl, jlv, kl} for its
proof.

\begin{theorem}
\label{hydrodynamics}
Fix $T>0$ and a measurable function $\ga:\T\to[0,1]$.  Let $\nu =
\nu_N$ be a sequence of probability measures on $X_N$ associated to
$\ga$, in the sense that
\begin{equation*}
\lim_{N\to\infty}
\nu_N \Big(|\lan\pi^{N},G\ran -\int_{\T} G(u)\ga(u)du|>\de\Big)\;=\;0\;,
\end{equation*}
for every $\de>0$ and every continuous function $G:\T\to\R$.  Then,
for every $t\ge 0$, every $\de>0$ and every continuous function
$G:\T\to\R$, we have
\begin{equation*}
\lim_{N\to\infty}
\mathbb{P}^{N}_{\nu}
\Big(|\lan\pi_{t}^{N},G\ran -\int_{\T} G(u)\rho(t,u)du|>\de\Big)\;=\;0\;,
\end{equation*}
where $\rho:[0,\infty)\times\T\to[0,1]$ is the unique weak solution of
the Cauchy problem
\begin{equation}
\label{rdeq}
\begin{cases}
     \partial_t\rho \;=\; (1/2) \De \rho + F(\rho) \ \text{ on }\ \T\;,\\
     \rho(0,\cdot)\;=\;\ga(\cdot)\;,
\end{cases}
\end{equation}
where $F(\rho)=B(\rho)-D(\rho)$.
\end{theorem}

The definition, existence and uniqueness of weak solutions of the
Cauchy problem \eqref{rdeq} are discussed in Section 6.

\subsection{Hydrostatic limit}
\label{hs}

We examine in this subsection the asymptotic behavior of the empirical
measure under the stationary state.  Fix $N\ge 1$ large enough. Since the Markov
Process $\eta^{N}_{t}$ is irreducible and the cardinality of the state
space $X_N$ is finite, there exists a unique invariant probability
measure for the process $\eta_{t}^{N}$, denoted by $\mu_N$.  Let
$\mathcal{P}_N$ be the probability measure on $\M_+$ defined by
$\mathcal{P}_N = \mu_N \circ (\pi^N)^{-1}$.

For each $p\ge1$, let $L^p(\T)$ be the space of all real $p$-th integrable functions
$G:\T\to\R$ with respect to the Lebesgue measure: $\int_\T |G(u)|^p du
<\infty$. The corresponding norm is denoted by $\|\cdot\|_{p}$:
\begin{equation*}
\|G\|_{p}^{p}\;:=\;\int_\T |G(u)|^p du\;.
\end{equation*}
In particular, $L^2(\T)$ is a Hilbert space equipped with the inner
product
\begin{align*}
\lan G, H \ran \;=\; \int_\T G(u)H(u) du\;.
\end{align*}
For a function $G$ in $L^{2}(\T)$, we also denote by $\lan G \ran$
the integral of $G$ with respect to the Lebesgue measure:
$\lan G \ran :=\int_{\T} G(u) du$.

Let $\E$ be the set of all classical solutions of the semilinear
elliptic equation:
\begin{equation}\label{seeq}
(1/2)\De\rho + F(\rho) \;=\; 0 \ \text{ on }\ \T\;.
\end{equation}
Classical solution means a function $\rho:\T\to[0,1]$ in $C^{2}(\T)$
which satisfies the equation \eqref{seeq} for any $u\in\T$.
We sometimes identify $\E$ with the set of all absolutely continuous
measures whose density are a classical solution of \eqref{seeq}:
\begin{equation*}
\{\pi\in\M_+:\pi(du)=\rho(u)du,\  \rho 
\text{ is a classical solution of the equation \eqref{seeq}}\}.
\end{equation*}

\begin{theorem}
\label{hsl}
The measure $\mathcal{P}_N$ asymptotically concentrates on the set
$\E$. Namely, for any $\de>0$, we have
\begin{align*}
\lim_{N\to\infty}\mathcal{P}_N(\pi\in\M_+ : 
\inf_{\bar{\pi}\in\E} d(\pi , \bar{\pi})\ge\de) \;=\; 0\;.
\end{align*}
\end{theorem}

If the set $\mathcal{E}$ is a singleton, it follows from Theorem
\ref{hsl} that the sequence $\{\mathcal{P}_{N}:N\ge1\}$ converges:

\begin{corollary}
  Assume that there exists a unique classical solution
  $\overline{\rho}:\T\to[0,1]$ of the semilinear elliptic equation
  \eqref{seeq}.  Then $\mathcal{P}_{N}$ converges to the Dirac measure
  concentrated on $\overline{\rho}(u)du$ as $N\to\infty$.
\end{corollary}

\begin{remark}
  In \cite{dpv, dppv}, De Masi et al. examined the dynamics introduced
  above in the case of the double well potential $F(\rho) = -V'(\rho)=
  a (2\rho-1) -b (2\rho-1)^3$, $a$, $b>0$, which is symmetric around
  the density $1/2$.  They proved that, starting from a product
  measure with mean $1/2$, the unstable equilibrium of the ODE $\dot
  x(t) = - V'(x(t))$, the empirical density remains in a neighborhood
  of $1/2$ in a time scale of order $\log N$.  Bodineau and Lagouge in
  Subsection of \cite{bl2} conjectured that Theorem \ref{hsl} remains true if we
  replace $\mathcal{E}$ by the set of all stable equilibrium solutions
  of the equation \eqref{seeq}.  This conjecture is proved in
  \cite{flt} and follows from the large deviation principle for the
  sequence $\{\mathcal{P}_{N}: N\ge1\}$.
\end{remark}

\subsection{Dynamical large deviations}
\label{ldp}

Denote by $\mathcal{M}_{+,1}$ the closed subset of
$\mathcal{M}_+$ of all absolutely continuous measures with density
bounded by $1$:
\begin{equation*}
\mathcal{M}_{+,1}\;=\;\{\pi\in\mathcal{M}_+(\T):
\pi(du)=\rho(u)du,\ 0\le\rho(u)\le1\ a.e.\ u\in\T\}\;.
\end{equation*}

Fix $T>0$, and denote by $C^{m,n}([0,T]\times\T)$, $m,n$ in
$\N_{0}\cup\{\infty\}$, the set of all real functions defined on
$[0,T]\times\T$ which are $m$ times differentiable in the first
variable and $n$ times on the second one, and whose derivatives are
continuous.  Let $Q_{\eta}=Q^{N}_{\eta}$, $\eta\in X_{N}$, be the
probability measure on $D([0,T],\M_{+})$ induced by the measure-valued
process $\pi^{N}_{t}$ starting from $\pi^{N}(\eta)$.

Fix a measurable function $\ga:\T\to[0,1]$.  For each path
$\pi(t,du)=\rho(t,u)du$ in $D([0,T],\mathcal{M}_{+,1})$, define the
energy $\mathcal{Q}: D([0,T],\mathcal{M}_{+,1})\to[0,\infty]$ as
\begin{equation}
\label{i03}
\mathcal{Q}(\pi)\;=\;\sup_{G\in C^{0,1}([0,T]\times\T)}
\Big \{2\int_0^Tdt\ \lan\rho_t,\nabla G_t\ran
-\int_0^Tdt\int_\T du\ G^2(t,u) \Big\}\;.
\end{equation}
It is known that the energy $\mathcal{Q}(\pi)$ is finite if and only
if $\rho$ has a generalized derivative and this generalized derivative
is square integrable on $[0,T]\times\T$:
\begin{equation*}
\int_{0}^{T} dt \ \int_{\T} du \ |\nabla\rho(t,u)|^{2} <\infty\;.
\end{equation*}
Moreover, it is easy to see that the energy $\mathcal{Q}$
is convex and lower semicontinuous.

For each function $G$ in $C^{1,2}([0,T]\times\T)$, define the functional
$\bar{J}_G:D([0,T],\mathcal{M}_{+,1})\to\R$ by
\begin{align*}
\bar{J}_G(\pi) & \; =\;\lan\pi_T,G_T\ran -\lan\ga,G_0\ran-
\int_0^Tdt\ \lan\pi_t, \partial_tG_t+\frac{1}{2}\De G_t\ran \\
&-\frac{1}{2}\int_0^Tdt\ \lan \chi(\rho_t), (\nabla G_t)^2\ran 
-\int_0^Tdt\ \big\{\lan B(\rho_t), e^{G_t}-1\ran + \lan D(\rho_t)\;, 
e^{-G_t}-1 \ran \big\},
\end{align*}
where $\chi(r)=r(1-r)$ is the mobility.  Let
$J_G:D([0,T],\mathcal{M}_+)\to[0,\infty]$ be the functional defined by
\begin{equation*}
J_G(\pi) \;=\; \begin{cases}
     \bar{J}_G(\pi)  & \text{if $\pi\in D([0,T],\mathcal{M}_{+,1})$}\;, \\
     \infty  & \text{otherwise}\;.
\end{cases}
\end{equation*}
We define the large deviation rate function
$I_T(\cdot|\ga):D([0,T],\mathcal{M}_+)\to[0,\infty]$ as
\begin{equation}
\label{i04}
I_T(\pi|\ga) \;=\; \begin{cases}
     \sup{J_G(\pi)} &
     \text{if $\mathcal{Q}(\pi)<\infty$}\;, \\
     \infty  & \text{otherwise}\;,
\end{cases}
\end{equation}
where the supremum is taken over all functions $G$ in $C^{1,2}([0,T]\times\T)$.

We review here an explicit formula for the functional $I_T$ at
smooth trajectories obtained in Lemma 2.1 of \cite{jlv}.
Let $\rho$ be a function in $C^{2,3}([0,T]\times\T)$  with
$c\le\rho \le 1-c$, for some $0< c <1/2$,
and set $\pi(t,du)=\rho(t,u)du$.
Then there exists a unique solution $H\in C^{1,2}([0,T]\times\T)$
of the partial differential equation
\begin{align*}
\partial_t\rho \;=\; (1/2)\De\rho -\nabla(\chi(\rho)\nabla H)
+ B(\rho)e^H -D(\rho)e^{-H} \;,
\end{align*}
with some initial profile $\ga$.
In the case, $I_T(\pi|\ga)$ can be expressed as
\begin{align*}
I_T(\pi|\ga) \;=\; & \frac{1}{2}\int_0^Tdt\  \lan\chi(\rho_t), (\nabla H_t)^2\ran \\
& + \int_0^Tdt\  \lan B(\rho_t), f(H_t) \ran + \int_0^Tdt\  \lan D(\rho_t), f(-H_t) \ran \;,
\end{align*}
where $f(a)=1-e^a+ae^a$.

The following theorem is one of main results of this paper.

\begin{theorem}
\label{mt1}
Assume that the functions $B$ and $D$ are concave on $[0,1]$.  Fix
$T>0$ and a measurable function $\ga:\T\to[0,1]$.  Assume that a
sequence $\eta^{N}$ of initial configurations in $X_{N}$ is associated
to $\ga$, in the sense that
\begin{equation*}
\lim_{N\to\infty}\lan\pi^{N}(\eta^{N}), G\ran \;=\; \int_{\T}G(u)\ga(u)du
\end{equation*}
for every continuous function $G:\T\to\R$.  Then, the measure
$Q_{\eta^{N}}$ on $D([0,T],\mathcal{M}_{+})$ satisfies a large
deviation principle with the rate function $I_{T}(\cdot|\ga)$. That
is, for each closed subset $\mathcal{C}\subset
D([0,T],\mathcal{M}_{+})$,
\begin{equation*}
\varlimsup_{N\to\infty}\frac{1}{N}\log{Q_{\eta^{N}}(\mathcal{C})} \;\le\; 
-\inf_{\pi\in\mathcal{C}}I_{T}(\pi|\ga)\;,
\end{equation*}
and for each open subset $\mathcal{O}\subset
D([0,T],\mathcal{M}_{+})$, 
\begin{equation*}
\varliminf_{N\to\infty}\frac{1}{N}\log{Q_{\eta^{N}}(\mathcal{O})} \;\ge\; 
-\inf_{\pi\in\mathcal{O}}I_{T}(\pi|\ga)\;.
\end{equation*}
Moreover, the rate function $I_{T}(\cdot|\ga)$ is lower semicontinuous
and has compact level sets.
\end{theorem}

\begin{remark}
  Jona-Lasinio, Landim and Vares \cite{jlv} proved the dynamical large
  deviations principle stated above, but the lower bound was obtained
  only for smooth trajectories. Bodineau and Lagouge \cite{bl} proved
  the lower bound for one-dimensional reaction-diffusion models in
  contact with reservoirs in the case where $B$ and $D$ are concave,
  monotone functions.
\end{remark}

\begin{remark}
  Proposition \ref{energy} asserts that there exists a finite constant
  $C_0$ such that if $\pi$ is a trajectory with finite energy,
  $Q(\pi)<\infty$, then $Q(\pi) \le C_0 (I_T(\pi|\gamma) +1)$. In the
  case where $B$ and $D$ are concave functions, we can use Theorem
  \ref{dense}, which asserts that the smooth trajectories are
  $I_T(|\gamma)$-dense, to prove the same bound without the assumption
  that the trajectory $\pi$ has finite energy. In particular, in this
  case we can define the rate function $I_T(\ |\gamma)$ simply as
  \begin{equation*}
    I_T(\pi|\ga) \;=\; \sup_G J_G(\pi)\;.
  \end{equation*}
\end{remark}

\begin{remark}
  In the proof that the rate function $I_{T}(\cdot|\ga)$ is lower
  semicontinuous and has compact level sets we do not use a bound on
  the $H_{-1}$ norm of $\partial_t \rho$ in terms of its rate function
  $I_T(\pi|\ga)$. Actually, as mentioned in the introduction, such
  a bound does not hold for reaction-diffusion models. Therefore,
  the arguments presented here permit to simplify the proof of the
  regularity of the rate function in other models, such as the weakly
  asymmetric simple exclusion process \cite{blm2, flm}.
\end{remark}

\section{Proof of Theorem \ref{hsl}}

We prove in this section Theorem \ref{hsl}.  Our approach is a
generalization of the one developed in \cite{flm, mo}, but it does not
require the existence of a global attractor for the underlying
dynamical system.  The method can be applied to any dynamics which
fulfills two conditions: the macroscopic evolution of the empirical
measure is described by a hydrodynamic equation, and for any initial
condition the solution of this equation converges to a stationary
profile as time goes to infinity. For instance, the boundary driven
reaction-diffusion models examined in \cite{bl}.

Recall from Subsection \ref{hs} the definition of the measure $\mu_N$
on $X_N$, the map $\pi^N$ from $X_N$ to $\M_+$ and the measure
$\mathcal{P}_N= \mu_N \circ (\pi^N)^{-1}$ on $\mathcal{M}_+$.  Denote
by $\mathbf{Q}^N$ the probability measure on the Skorokhod space
$D([0,\infty), \M_+)$ induced by the measure-valued process
$\pi^N_{t}$ under the initial distribution $\mathcal{P}_{N}$. Since
the measure $\mu_N$ is stationary under the dynamics, $\mathcal{P}_N
(\mathcal{B}) = \mathbf{Q}^N (\pi:\pi_T\in\mathcal{B})$, for each
$T>0$ and Borel set $\mathcal{B}\subset\mathcal{M}_+$.

\begin{lemma}
\label{tight}
The sequence $\{\mathbf{Q}^N:N\ge1\}$ is tight and all its limit
points $\mathbf{Q}^*$ are concentrated on absolutely continuous
paths $\pi(t,du) = \rho(t,u)du$ whose density $\rho$ is nonnegative
and bounded above by $1$ :
\begin{gather*}
\mathbf{Q}^*\{ \pi:\pi(t,du) = \rho(t,u)du\ ,\ \text{for}\  t\in[0,\infty)\}
\;=\; 1\;, \\
\mathbf{Q}^*\{ \pi:0 \le \rho(t,u)\le 1\ ,\ 
\text{for}\  (t,u)\in[0,\infty)\times\T\} \;=\; 1\;.
\end{gather*}
\end{lemma}

The proof of this lemma is similar to the one of Proposition 3.1 in
\cite{lms}.  

Let $\mathcal{A}$ be the set of all trajectories
$\pi(t,du)=\rho(t,u)du$ in $D([0,\infty),\M_{+,1})$ whose density $\rho$
is a weak solution to the Cauchy problem \eqref{rdeq} for
some initial profile $\rho_{0}:\T\to[0,1]$.

\begin{lemma}
\label{conc}
All limit points $\mathbf{Q}^*$ of the sequence $\{\mathbf{Q}^N:N\ge1\}$ are 
concentrated on paths $\pi(t,du) = \rho(t,u) du$ in $\mathcal{A}$ :
\begin{equation*}
\mathbf{Q}^*(\mathcal{A}) \;=\; 1\;.
\end{equation*}
\end{lemma}

The proof of this lemma is similar to the one of Lemma A.1.1 in \cite{kl}.

\begin{proof}[Proof of Theorem \ref{hsl}]
Fix a positive $\de>0$. Let $\mathcal{E}_\de$ be the
$\de$-neighborhood of $\mathcal{E}$ in $\mathcal{M}_+$ :
\begin{align*}
\mathcal{E}_\de \;:=\; \{\pi\in\M_+:\inf_{\bar{\pi}\in\E} 
d(\pi , \bar{\pi}) < \de\}\;.
\end{align*}
Denote by $\mathcal{E}_\de^c$ the complement of the set
$\mathcal{E}_\de$.  The assertion of Theorem \ref{hsl} can be
rephrased as
\begin{equation*}
\lim_{N\to\infty}\mathcal{P}_N(\mathcal{E}_\de^c) \;=\; 0\;.
\end{equation*}
Therefore, to conclude the theorem it is enough to show that any limit
point of the sequence $\mathcal{P}_N(\mathcal{E}_\de^c)$ is equal to
zero.

Fix $T>0$. Since the measure $\mu_N$ is invariant under the dynamics,
\begin{equation}
\label{eq3.1}
\mathcal{P}_N(\mathcal{E}_\de^c) \;=\; 
\mathbf{Q}^N(\pi: \pi_T\in\mathcal{E}_\de^c)\;.
\end{equation}
Let $\mathbf{Q}^{*}$ be a limit point of $\{\mathbf{Q}^{N}:N\ge1\}$
and take a subsequence $N_k$ so that the sequence
$\{\mathbf{Q}^{N_{k}}:k\ge1\}$ converges to $\mathbf{Q}^{*}$ as
$k\to\infty$.  Note that the set $\{\pi:
\pi_{T}\in\mathcal{E}_{\de}^{c}\}$ is not closed in
$D([0,\infty),\M_{+})$. However, we claim that
\begin{equation}
\label{bound3.1}
\varlimsup_{k\to\infty}\mathbf{Q}^{N_{k}}(\pi: \pi_T\in\mathcal{E}_\de^c)
\;\le\; \mathbf{Q}^{*}(\{\pi : \pi_T\in\mathcal{E}_\de^c\}\cap\mathcal{A})\;,
\end{equation}
where $\mathcal{A}$ is the set introduced just before Lemma
\ref{conc}.  Indeed, denote by $\overline{\{\pi:
  \pi_{T}\in\mathcal{E}_{\de}^{c}\}}$ the closure of the set $\{\pi:
\pi_{T}\in\mathcal{E}_{\de}^{c}\}$ under the Skorokhod topology. By
definition of the weak topology and by Lemma \ref{conc}, 
\begin{align*}
\varlimsup_{k\to\infty} \mathbf{Q}^{N_{k}}(\pi: \pi_T\in\mathcal{E}_\de^c)
\;\le\; \mathbf{Q}^{*}(\overline{\{\pi: \pi_T\in\mathcal{E}_\de^c\}}) 
\;=\; \mathbf{Q}^{*}(\overline{\{\pi:
  \pi_T\in\mathcal{E}_\de^c\}}\cap\mathcal{A})\; .
\end{align*}
It remains to prove that
\begin{equation*}
\overline{\{\pi: \pi_T\in\mathcal{E}_\de^c\}}\cap\mathcal{A}
\;=\; \{\pi: \pi_T\in\mathcal{E}_\de^c\}\cap\mathcal{A}\;.
\end{equation*}
Let $\pi$ be a path in $\overline{\{\pi:
  \pi_T\in\mathcal{E}_\de^c\}}\cap\mathcal{A}$.  Then there exists a
sequence $\{\pi^{n}:n\ge1\}$ such that $\pi^{n}$ converges to $\pi$ in
$D([0,\infty),\M_{+})$ as $n\to\infty$ and $\pi^{n}_{T}$ belongs to
$\mathcal{E}_{\de}^{c}$ for any $n\ge1$.  Since $\mathcal{A}$ is
contained in $C([0,\infty),\M_{+,1})$, the sequence
$\{\pi^{n}:n\ge1\}$ converges to $\pi$ under the uniform topology.
Hence $\pi_{T}^{n}$ converges to $\pi_{T}$. Since
$\mathcal{E}_{\de}^{c}$ is closed in $\M_{+}$, $\pi_{T}$ also belongs
to $\mathcal{E}_{\de}^{c}$, which proves \eqref{bound3.1}.

Fix a path $\pi(t,du)=\rho(t,u)du$ in $\mathcal{A}$.  By Proposition
\ref{conv}, there exists a density profile $\rho_{\infty}$ in $\E$
such that $\rho_{t}$ converges to $\rho_{\infty}$ in $C^{2}(\T)$.
Hence,
\begin{equation}
\label{conv3.1}
\mathcal{A} \;\subset\; \bigcup_{j\ge 1} \bigcap_{k\ge j} 
\{\pi_k \in\mathcal{E}_\de\}\;.
\end{equation}
By \eqref{eq3.1} and \eqref{bound3.1}, 
\begin{equation*}
\varlimsup_{N\to\infty}\mathcal{P}_N(\mathcal{E}_\de^c) \;\le\;
\mathbf{Q}^{*}(\{\pi: \pi_k\in\mathcal{E}_\de^c\}\cap\mathcal{A})
\quad\text{ for all } k\ge 1\;.
\end{equation*}
Since this bound holds for any $k\ge 1$,
\begin{equation*}
\varlimsup_{N\to\infty}\mathcal{P}_N(\mathcal{E}_\de^c) \;\le\;
\varlimsup_{k\to\infty} \mathbf{Q}^{*}
(\{\pi_k\in\mathcal{E}_\de^c\}\cap\mathcal{A})
\;\le\; \mathbf{Q}^{*}
\Big(\bigcap_{j\ge 1} \bigcup_{k\ge j}
\{\pi_k\in\mathcal{E}_\de^c\}\cap\mathcal{A}\Big)\;.
\end{equation*}
This latter set is empty in view of \eqref{conv3.1}, which completes
the proof of the theorem.
\end{proof}

\section{The rate function $I_{T}(\cdot|\ga)$}
\label{sec4}

We prove in this section that the large deviations rate function is
lower semicontinuous and has compact level sets.  These properties
play a fundamental role in the proof of the static large deviation
principle, cf. \cite{bg, f}. One of the main steps in the proof of
these properties is Proposition \ref{energy}. It asserts that there
exists a finite constant $C_0$ such that for all trajectory $\pi(t,du)
= \rho(t,u)$ whose density $\rho$ has finite energy, we have
$Q(\pi) \le C_0 (I_{T}(\pi|\ga) +1)$. Such bound was first proved in 
\cite{qrv}.

\begin{proposition}
\label{cont}
Let $\pi$ be a path in $D([0,T],\M_{+})$ such that $I_{T}(\pi|\ga)$ is
finite.  Then $\pi(0,du)=\ga(u)du$ and $\pi$ belongs to
$C([0,T],\M_{+,1})$.
\end{proposition}

\begin{proof}
The proof of this proposition is similar to the one of Lemma 3.5 in
\cite{bdgjl2}.  Actually, the computation performed in the proof of Lemma
3.5 in \cite{bdgjl2} gives that, for any $g$ in $C^{2}(\T)$ and any $0\le
s < t\le T$,
\begin{align}
\label{bound4.1}
|\lan\pi_{t},g\ran - \lan\pi_{s}, g\ran | \;\le\;
C \a_{s,r}\{ I_{T}(\pi|\ga) +1\}\;,
\end{align}
for some positive constant $C=C(g)$, which depends only on $g$.
In the inequality \eqref{bound4.1}, the constant $\a_{s,r}$
is given by $(\log{(r-s)^{-1}})^{-1}$.
\eqref{bound4.1} implies the desired continuity.
\end{proof}

The next proposition plays an important role in the proof of Theorem
\ref{lsc}.

\begin{proposition}
\label{energy}
There exists a constant $C_0>0$ such that, for any path
$\pi(t,du)=\rho(t,u)du$ in $D([0,T],\mathcal{M}_{+,1})$ with finite
energy, we have
\begin{equation*}
\int_0^{T} dt\ \int_{\T} du\ \frac{|\nabla\rho(t,u)|^{2}}{\chi(\rho(t,u))}
\;\le\; C_0\,\{ I_{T}(\pi|\ga)+1\}\;.
\end{equation*}
\end{proposition}

We fix some notation before proving Proposition \ref{energy}.

Let $H^{1}(\T)$ be the Sobolev space of functions $G$ with generalized
derivatives $\nabla G$ in $L^{2}(\T)$. $H^{1}(\T)$ endowed with the
scalar product $\lan\cdot,\cdot\ran_{1,2}$, defined by
\begin{equation*}
\lan G, H\ran_{1,2}\;=\;\lan G, H\ran + \lan \nabla G, \nabla H\ran\;,
\end{equation*}
is a Hilbert space. The corresponding norm is denoted by $\|\cdot\|_{1,2}$:
\begin{equation*}
\|G\|_{1,2}^{2}\;:=\;\int_\T |G(u)|^2 du + \int_\T |\nabla G(u)|^2 du\;.
\end{equation*}
For a Banach space $(\mathbb{B}, \|\cdot\|_{\mathbb{B}})$ and $T>0$,
we denote by $L^{2}([0,T], \mathbb{B})$ the Banach space of measurable
functions $U:[0,T]\to\mathbb{B}$ for which
\begin{equation*}
\|U\|^{2}_{L^{2}([0,T],\mathbb{B})}
\;=\;\int_{0}^{T} \|U_{t}\|_{\mathbb{B}}^{2}\  dt < \infty
\end{equation*}
holds. For each $p\ge1$ and $T>0$, let $L^p([0,T]\times\T)$
be the space of all real $p$-th integrable functions
$U:[0,T]\times\T\to\R$ with respect to the Lebesgue measure: 
$\int_{0}^{T}dt\int_\T |U(t,u)|^p du < \infty$.

Fix a path $\pi(t,du)=\rho(t,u)du$ in $D([0,T],\M_{+,1})$ with finite
energy.  For a smooth function $G: [0,T]\times\bb T\to\bb R$ and for
a bounded function $H$ in $L^{2}([0,T],H^{1}(\T))$, define the
functionals
\begin{align*}
& L_{G}(\pi) \;=\; \lan \pi_{T}, G_{T}\ran - \lan \pi_0 , G_{0}\ran -
\int_{0}^{T} dt \ \lan \pi_{t}, \partial_{t}G_{t} \ran\;, \\
&\quad B^{1}_{H}(\pi) \;=\; \frac{1}{2}\int_{0}^{T} dt \ 
\lan\nabla\rho_{t}, \nabla H_{t}\ran
\;-\; \frac{1}{2}\int_{0}^{T} dt \ 
\lan \chi(\rho_{t}), (\nabla H_{t})^{2}\ran\;, \\
&\qquad B^{2}_{H}(\pi) \;=\;
\int_0^Tdt\ \Big\{\lan B(\rho_t), e^{H_t}-1\ran 
+ \lan D(\rho_t), e^{-H_t}-1 \ran \Big \}\;.
\end{align*}
Note that, for paths $\pi(t,du)$ such that $\pi(0,du) = \gamma(u) du$,
\begin{equation}
\label{var}
\sup_{H\in C^{1,2}([0,T]\times\T)}
\Big\{ L_{H}(\pi)+B_{H}^{1}(\pi)-B_{H}^{2}(\pi)\Big\} 
\;=\; I_{T}(\pi|\ga)\;.
\end{equation}

Consider the function $\phi:\R\to[0,\infty)$ defined by
\begin{align*}
\phi(r)\;:=\;
\begin{cases}
\dfrac{1}{Z}\exp{\{-\dfrac{1}{(1-r^{2})}\}}\ & \text{ if } |r|<1\;, \\
\ 0\ & \text{otherwise}\;,
\end{cases}
\end{align*}
where the constant $Z$ is chosen so that $\int_{\R}\phi(r) dr =1$.
For each $\de>0$, let
\begin{equation*}
\phi^{\de}(r) \;:=\; \dfrac{1}{\de}\,
\phi \Big(\dfrac{r}{\de}\Big)\;.
\end{equation*}
Since the support of the function $\phi^{\de}$ is contained in
$[-\de,\de]$, the function $\phi^{\de}$ can be regarded as a function
on $\T$.  To distinguish convolution in time from convolution in
space, we denote by $\psi^{\de}:\T\to[0,\infty)$ the function
$\phi^{\e}$ defined on $\T$ with $\e=\de$.

Denote by $f*g$ the space or time convolution of two functions $f$,
$g$:
\begin{equation*}
(f*g)(a) \;=\; \int f(a-b)\, g(b) \, db\;,
\end{equation*}
where the integral runs over $\bb R$ in the case where $f$, $g$ are
functions of time and over $\bb T$ in the case where $f$ and $g$ are
functions of space.

Throughout this section, we adopt the following notation: For a
bounded measurable function $\rho:[0,T]\times\T\to\R$, define the
smooth approximation in space, time and space-time by
\begin{align*}
& \rho^{\e}(t,u) \;:=\; [\rho(t,\cdot)*\psi^{\e}] (u) \;=\;
\int_{\T}\rho(t,u+v)\psi^{\e}(v) dv\;, \\
&\quad \rho^{\de}(t,u) \;:=\; [\rho(\cdot, u)*\phi^{\de}](t)
\;=\; \int_{-\de}^{\de}\rho(t+r,u)\phi^{\de}(r)dr\;, \\
&\qquad \rho^{\e, \de}(t,u) \;:=\; \int_{-\de}^{\de} dr\ 
\int_{\T} dv \ \rho(t+r,u+v) \psi^{\e}(v) \phi^{\de}(r)\;.
\end{align*}
In the above formulas, we extend the definition of $\rho$ to
$[-1,T+1]$ by setting $\rho_{t}=\rho_{0}$ for $-1\le t \le 0$ and
$\rho_{t}=\rho_{T}$ for $T \le t \le T+1$.  Remark that we use similar
notation, $\rho^{\e}$ and $\rho^{\de}$, for different
objects. However, $\rho^{\e}$ and $\rho^{\de}$ always represent a
smooth approximation of $\rho$ in space and time, respectively.  For
each $\pi(t,du)=\rho(t,u)du$, we also define paths
$\pi^{\e}(t,du)=\rho^{\e}(t,u)du$, $\pi^{\de}(t,du)=\rho^{\de}(t,u)du$
and $\pi^{\e,\de}(t,du)=\rho^{\e,\de}(t,u)du$.

We summarize some properties of $\rho^{\e}$ in the next proposition.
The proof is elementary and is thus omitted.

\begin{proposition}
\label{prope}
Let $\rho:[0,T]\times\T\to\R$ be a function in $L^{2}([0,T],
H^{1}(\T))$.  Then, for each $\e>0$, $\rho^{\e}$ and $\nabla\rho^{\e}$
converges to $\rho$ and $\nabla\rho$ in $L^{2}([0,T]\times\T)$,
respectively.  Moreover, if $\rho$ is bounded in $[0,T]\times\T$ and
the application $\lan\rho_{t}, g\ran$ is continuous on the time
interval $[0,T]$ for any function $g$ in $C^{\infty}(\T)$, then, for
each $\e>0$, $\rho^{\e}$ is uniformly continuous on $[0,T]\times\T$.
\end{proposition}

For each $a>0$, define the functions $h=h_a$ and $\chi_{a}$ on $[0,1]$
by
\begin{gather*}
h(\rho) \;:=\; \dfrac{1}{2(1+2a)} \, \Big\{ (\rho+a)
\log{(\rho+a)}+(1-\rho+a)\log{(1-\rho+a)} \Big\}\;, \\
\chi_{a}(\rho) \;:=\; (\rho+a)(1-\rho+a)\;.
\end{gather*}
Note that $h''=(2\chi_{a})^{-1}$.

Until the end of this section, $0<C_0<\infty$ represents a constant
independent of $\e$, $\de$ and $a$ and which may change from line to
line.

\begin{lemma}
\label{decom}
Let $R^{\e,\de}$ be the difference between $L_{H}(\pi^{\e,\de})$ and
$L_{H^{\e,\de}}(\pi)$:
\begin{equation*}
R^{\e,\de} \;=\; L_{H}(\pi^{\e,\de}) \;-\; L_{H^{\e,\de}}(\pi)\;,
\end{equation*}
where $H = h_a'(\rho^{\e,\de})$. Then, for any fixed $\e>0$,
$R^{\e,\de}$ converges to $0$ as $\de\downarrow0$.
\end{lemma}

\begin{proof}
Keep in mind that $H = h_a'(\rho^{\e,\de})$ depends on $\e$ and $\de$,
although this does not appears in the notation, and recall that $C_0$
represents a constant independent of $\e$, $\de$ and $a$ which may
change from line to line.
A change of variables shows that
\begin{align*}
L_{H}(\pi^{\e,\de}) 
& \;=\;\lan \rho^{\de}_{T}, H_{T}^{\e} \ran
- \lan\rho_{0}^{\de}, H_{0}^{\e} \ran
- \int_{0}^{T} dt \ \lan \rho^{\de}_{t}, \partial_{t}H_{t}^{\e}\ran  \\
& \;=\; \lan \rho_{T}, H_{T}^{\e,\de} \ran
- \lan\rho_{0}, H_{0}^{\e,\de} \ran
- \int_{0}^{T} dt \ \lan \rho^{\de}_{t}, \partial_{t}H_{t}^{\e}\ran + R_{1}^{\e,\de}\;,
\end{align*}
where
\begin{equation*}
R_{1}^{\e,\de} := R^{\e,\de,T} - R_{0}^{\e,\de,0} \quad\text{and}\quad
R^{\e,\de,t} :=\lan \rho_{t}^{\de}-\rho_{t}, H_{t}^{\e}\ran 
+\lan \rho_{t}, H_{t}^{\e}-H_{t}^{\e,\de}\ran\;
\end{equation*}
for $0\le t\le T$.

From a simple computation it is easy to see that
\begin{equation*}
\int_{0}^{T}dt \  \lan \rho_{t}^{\de}, \partial_{t}H_{t}^{\e}\ran
\;=\; \int_{0}^{T} dt\ \lan\rho_{t}, \partial_{t}H_{t}^{\e,\de} \ran +
R_{2}^{\e,\de}\;, 
\end{equation*}
where $|R_{2}^{\e,\de}|\le C_0\de\|\partial_{t}H^{\e}\|_{\infty}$.  To
conclude the proof, it is enough to show that, for each fixed $\e>0$,
$R_{1}^{\e,\de}$ and $\de\|\partial_{t}H^{\e}\|_{\infty}$ converge to
zero as $\de\downarrow0$.  

Fix $\e>0$.  We first prove that
\begin{equation}
\label{lim4.1}
\lim_{\de\downarrow0}R^{\e,\de,t} \;=\; 0 \quad\text{for}\quad 
t=0\;\text{ and } \;t=T\;.
\end{equation}
We prove this assertion for $t=T$, the argument being similar for
$t=0$. A change of variables shows that
\begin{equation*}
R^{\e,\de,T} \;=\; \lan \rho_{T}^{\e,\de}-\rho_{T}^{\e}, H_{T}\ran 
+\lan \rho^{\e}_{T}, H_{T}-H_{T}^{\de}\ran \;.
\end{equation*}
By Proposition \ref{prope}, $\rho^{\e}(\cdot,u)$ is continuous for any
$u\in\T$.  Therefore, for any $(t,u)\in[0,T]\times\T$,
\begin{equation}
\label{lim4.4}
\begin{gathered}
\lim_{\de\downarrow0}\rho^{\e,\de}(t,u) \;=\; \rho^{\e}(t,u) \; ,  \\
\lim_{\de\downarrow0} H^{\de}(T,u)
\;=\; h_a'(\rho^{\e}(T,u)) \;=\;
\lim_{\de\downarrow0}H(T,u)\;.
\end{gathered}
\end{equation}
Since $h'$ is bounded and continuous on $[0,1]$,
\eqref{lim4.1} is proved by letting $\de\downarrow0$
and by the bounded convergence theorem.

It remains to show that $\de\|\partial_{t}H^{\e}\|_{\infty}$ converges
to $0$ as $\de\downarrow0$.  An elementary computation gives that, for
any $(t,u)\in[0,T]\times\T$,
\begin{equation*}
\partial_{t}H^{\e}(t,u)
\;=\;\int_{\T} dv\ h''(\rho^{\e,\de}(t,u+v))\psi^{\e}(v) 
\int_{-\de}^{\de} dr\ 
\rho^{\e}(t+r,u+v)(\phi^{\de})'(r)\;.
\end{equation*}
Since $\phi^{\de}$ is a symmetric function, a change of variables
shows that
\begin{equation*}
\int_{-\de}^{\de} dr\ 
\rho^{\e}(t+r,u+v)(\phi^{\de})'(r)
\;=\;\int_{-\de}^{0} dr\ 
\{\rho^{\e}(t+r,u+v)-\rho^{\e}(t-r,u+v)\}(\phi^{\de})'(r)\;.
\end{equation*}
By Proposition \ref{prope}, $\rho^{\e}$ is uniformly continuous on
$[-1,T+1]\times \bb T$. On the other hand, $\de\int_{-\de}
^{0}(\phi^{\de})'(r) dr = \phi(0)$.  Therefore, the last expression
multiplied by $\de$ converges to $0$ as $\de\downarrow0$ uniformly in
$(t,u)\in[0,T]\times\T$. Since $h''$ and $\psi^{\e}$ are uniformly
bounded, $\de\|\partial_{t}H^{\e}\|_{\infty}$ converges to $0$ as
$\de\downarrow0$.
\end{proof}

\begin{lemma}
\label{onB}
For any path $\pi(t,du)=\rho(t,u)du$ such that $\mathcal{Q}(\pi)<\infty$ and
for $i=1,2$, 
\begin{equation*}
\lim_{\e\downarrow0}\lim_{\de\downarrow0}
B^{i}_{H^{\e,\de}}(\pi) \;=\; B^{i}_{h'(\rho)}(\pi)\;.
\end{equation*}
Moreover, there exists a positive constant $C_0<\infty$, independent
of $a>0$, such that
\begin{equation}
\label{bound4.4}
\int_{0}^{T} dt \ \int_{\T} du\ \frac{(\nabla\rho(t,u))^{2}}{\chi_{a}(\rho(t,u))}
\;\le\; C_0 \, B^{1}_{h'(\rho)}(\pi)\;,\quad
|B^{2}_{h'(\rho)}(\pi)| \;\le\; C_0\;.
\end{equation}
\end{lemma}

\begin{proof}
Throughout this proof, $C(a)$ expresses a constant depending only on $a>0$
which may change from line to line.

Let $\pi(t,du)=\rho(t,u)du$ be a path in $D([0,T], \M_{+,1})$ such that
$Q(\pi)<\infty$.  We first show that
\begin{equation}
\label{lim4.2}
\lim_{\e\downarrow0}\lim_{\de\downarrow0}
B^{1}_{H^{\e,\de}}(\pi) \;=\; B^{1}_{h'(\rho)}(\pi)\;.
\end{equation}

Since $\nabla\rho^{\e}=\rho*\nabla\psi^{\e}$, by Proposition \ref{prope},
$\nabla\rho^{\e}$ is uniformly continuous in $[0,T]\times \T$.
Therefore, for any $(t,u)\in[0,T]\times\T$, we have
\begin{gather*}
\lim_{\de\downarrow0}\nabla\rho^{\e,\de}(t,u)\;=\;\nabla\rho^{\e}(t,u)\;, \\
\lim_{\de\downarrow0}\nabla H^{\e,\de}(t,u)
\;=\;\int_{\T} dv\ \psi^{\e}(v) h_a''(\rho^{\e}(t,u+v))\nabla\rho^{\e}(t,u+v)\;.
\end{gather*}
Hence, by the bounded convergence theorem and a change of variables,
\begin{equation}
\label{01}
\lim_{\de\downarrow0}B^{1}_{H^{\e,\de}}(\pi) \;=\; 
\dfrac{1}{2}\int_{0}^{T} dt\ 
\big\{\lan \nabla \rho^{\e}_{t}, h_a''(\rho^{\e}_{t})\nabla\rho^{\e}_{t}\ran
- \lan\chi(\rho_{t}), ([h_a''(\rho^{\e}_t)\nabla
\rho^{\e}_t]^{\e})^{2} \ran \big\}\;.
\end{equation}

On the one hand, since for any fixed $a>0$ $h_a''$ is bounded, and
since by Proposition \ref{prope}, $\nabla\rho^{\e}$ converges to
$\nabla\rho$ in $L^{2}([0,T]\times\T)$,
\begin{equation*}
\lim_{\e\downarrow0} \int_{0}^{T} dt\ 
\big\lan h_a''(\rho^{\e}_{t}) \, \big[
\nabla\rho^{\e}_{t} - \nabla\rho_{t}\big]^2 \big\ran \;=\; 0\;.
\end{equation*}
As $\rho$ has finite energy and $h_{a}''$ is bounded, the family $\{
h_{a}''(\rho^{\e})[\nabla\rho]^2; \e>0\}$ is uniformly
integrable. Moreover, since $h_{a}''$ is Lipschitz continuous, by
Proposition \ref{prope}, $h_{a}''(\rho^{\e})$ converges to
$h_{a}''(\rho)$ as $\e\downarrow0$ in measure, that is, for any $b>0$,
the Lebesgue measure of the set $\{(t,u)\in[0,T]\times\T;
|h_{a}''(\rho^{\e}(t,u))-h_{a}''(\rho(t,u))|\ge b\}$ converges
to $0$ as $\e\downarrow0$.  Therefore
\begin{equation}\label{lim4.5}
\lim_{\e\downarrow0} \int_{0}^{T} dt\ 
\big\lan h_a''(\rho^{\e}_{t}) \, \big[\nabla\rho_{t}\big]^2 \big\ran
\;=\; \int_{0}^{T} dt\ 
\big\lan h_a''(\rho_{t}) \, \big[\nabla\rho_{t}\big]^2 \big\ran
\;.
\end{equation}
On the other hand, by Schwarz inequality,
\begin{align*}
& \limsup_{\e\downarrow 0}\int_{0}^{T} dt \, \big\lan\chi(\rho_{t})\, 
\big\{ [h_a''(\rho^{\e}_t)\nabla
\rho^{\e}_t - h_a''(\rho_t)\nabla \rho_t]^{\e}  \big\}^{2}
\big\ran \\
& \quad \;\le\; 
\limsup_{\e\downarrow 0}\int_{0}^{T} dt \, \big\lan\chi(\rho_{t})\, 
\big\{ h_a''(\rho^{\e}_t)\nabla
\rho^{\e}_t - h_a''(\rho_t)\nabla \rho_t  \big\}^{2} \big\ran \;. 
\end{align*}
We may now repeat the arguments presented to estimate the first term
on the right hand side of \eqref{01} to show that the last expression
vanishes. 

Since $\chi$ is a bounded function, to complete the proof of
\eqref{lim4.2}, it remains to show that
\begin{equation*}
\limsup_{\e\downarrow 0} \int_{0}^{T} dt \, \big\lan
\big\{ [h_a''(\rho_t)\nabla \rho_t]^{\e}  - h_a''(\rho_t)\nabla \rho_t
\big\}^{2} \big\ran \;=\;0\;.
\end{equation*}
We estimate the previous integral by the sum of two terms, the first
one being
\begin{align*}
& \int_{0}^{T} dt \, \big\lan \big\{ [h_a''(\rho_t)\nabla \rho_t]^{\e}  
- [h_a''(\rho_t)]^\e \nabla \rho_t \big\}^{2} \big\ran \\
&\quad \;\le\; 
C(a) \int_{0}^{T} dt \, \int_{\bb T} dv \, \psi^\e(v) \, 
\big\lan \big\{ \nabla \rho_t (u+v)  
-  \nabla \rho_t (u) \big\}^{2} \big\ran\;,
\end{align*}
where we used Schwarz inequality and the fact that $h_a''$ is
uniformly bounded. This expression vanishes as $\e\to 0$ because
$\nabla \rho$ belongs to $L^2([0,T]\times\bb T)$. The second term in
the decomposition is
\begin{align}\label{lim4.6}
\int_{0}^{T} dt \, \big\lan [\nabla \rho_t]^{2}  \big\{ [h_a''(\rho_t)]^{\e}  
- h_a''(\rho_t) \big\}^{2} \big\ran \;.
\end{align}
By the argument leading to \eqref{lim4.5}, the expression
\eqref{lim4.6} converges to $0$ as $\e\downarrow0$.

We turn to the proof that
\begin{equation}
\label{lim4.3}
\lim_{\e\downarrow0}\varlimsup_{\de\downarrow0}\,
|B^{2}_{H^{\e,\de}}(\pi) - B^{2}_{h'(\rho)}(\pi)| \;=\; 0\;.
\end{equation}
Since $B, D$ and $h'$ are bounded functions, the difference
appearing in the previous formula is less than or equal to
\begin{align*}
& C(a) \Big\{\int_{0}^{T} \, \|e^{H^{\e, \de}_{t}}
-e^{h'(\rho_{t})}\|_{1} \ dt
\;+\;  \int_{0}^{T} \,  
\|e^{-H^{\e, \de}_{t}} -e^{-h'(\rho_{t})}\|_{1} \ dt \Big\} \\
& \le\; C(a) \int_{0}^{T} \,   \,
\|H^{\e, \de}_{t} - h'(\rho_{t})\|_{1}\ dt  \,  \;.
\end{align*}
By Proposition \ref{prope}, $\rho^{\e}$ is uniformly continuous in
$[0,T]\times\T$.  Therefore letting $\delta\to0$, the previous
expression converges to
\begin{align*}
& C(a) \int_{0}^{T} dt\, \, 
\| \, [h'(\rho^\e_{t})]^\e - h'(\rho_{t}) \|_{1} \ dt \, \\
&\quad \le\; C(a) \Big\{ \int_{0}^{T} \ \, \|
[h'(\rho^{\e}_t)]^{\e} - h'(\rho^{\e}_{t}) \|_{1} \, \ dt
+ \int_{0}^{T} \  \, \| h'(\rho^{\e}_{t}) -
h'(\rho_{t}) \|_{1} \  dt \Big\} \, \;.
\end{align*}
Since $h'$ is Lipschitz continuous and $\rho^{\e}$ converges to $\rho$
in $L^{2}([0,T]\times \T)$, the second integral vanishes in the limit
as $\e\downarrow0$.  On the other hand, the first integral is bounded
above by
\begin{align*}
& C(a) \int_{0}^{T} dt\  \int_{\T} dv\ \psi^{\e}(v)
\int_{\T} du\ | \rho_{t}^{\e}(u+v)-\rho_{t}^{\e}(u) |\,
 \\
&\quad \;\le\; C(a)  \int_{0}^{T} dt\  \int_{\T} dv\ \psi^{\e}(v)
\int_{\T} du\ | \rho_{t}(u+v)-\rho_{t}(u) | \;.
\end{align*}
This last integral vanishes in the limit as $\e\downarrow0$
because $\rho$ belongs to $L^{2}([0,T]\times\T)$.

To proof of the first bound in \eqref{bound4.4} is elementary and left
to the reader. To prove the second one, recall from \eqref{02} that
there exist polynomials $\tilde{B},\tilde{D}$ such that
$B(\rho)=(1-\rho)\tilde{B}(\rho)$ and $D(\rho)=\rho\tilde{D}(\rho)$.
From this fact, it is easy to see that the second bound in
\eqref{bound4.4} holds for some finite constant $C_0$, independent of
$a>0$.
\end{proof}

\begin{proof}[Proof of Proposition \ref{energy}]
We may assume, without loss of generality, that $I_{T}(\pi|\ga)$
is finite.  From the variational formula \eqref{var} and Lemma \ref{decom}, 
\begin{align}
\label{bound4.2}
L_{H}(\pi^{\e,\de})
+B^{1}_{H^{\e,\de}}(\pi)-B^{2}_{H^{\e,\de}}(\pi) - R^{\e,\de}
\;\le\; I_{T}(\pi|\ga)\;,
\end{align}
where $H$ stands for the function $h'(\rho^{\e,\de})$.  

Since $\rho^{\e,\de}$ is smooth, an integration by parts
yields the identity
\begin{equation*}
L_H(\pi^{\e,\de})\;=\;
\lan h(\rho^{\e,\de}_{T})\ran
-\lan h(\rho^{\e,\de}_{0})\ran\;.
\end{equation*}
There exists, therefore, a constant $C_0$, independent of $\e$, $\de$
and $a$, such that
\begin{equation*}
|L_H(\pi^{\e,\de})| \;\le\; C_0\;.
\end{equation*}

In \eqref{bound4.2}, let $\de\downarrow0$ and then $\e\downarrow0$. It
follows from the previous bound, and from Lemmas \ref{decom} and
\ref{onB} that
\begin{equation*}
\int_0^{T} dt\ \int_{\T} du\ \frac{|\nabla\rho(t,u)|^{2}}
{\chi_{a}(\rho(t,u))} \;\le\; C_0 \{ I_{T}(\pi|\ga)+1\}\;.
\end{equation*}
It remains to let $a\downarrow0$ and to use Fatou's lemma.
\end{proof}

\begin{corollary}
\label{zero}
The density $\rho$ of a path $\pi(t,du)=\rho(t,u)du$ in $D([0,T],
\M_{+,1})$ is the weak solution of the Cauchy problem \eqref{rdeq}
with initial profile $\ga$ if and only if the rate function
$I_{T}(\pi|\ga)$ is equal to $0$.  Moreover, in that case
\begin{equation}
\label{bound4.7}
\int_{0}^{T} dt\ \int_{\T}du \ 
\dfrac{|\nabla \rho(t,u)|^{2}}{\chi(\rho(t,u))}
< \infty.
\end{equation}
\end{corollary}

\begin{proof}
If the density $\rho$ of a path $\pi(t,du)=\rho(t,u)du$ in $D([0,T], \M_{+,1})$
is the weak solution of the Cauchy problem \eqref{rdeq},
then for any $G$ in $C^{1,2}([0,T]\times\T)$ we have
\begin{align*}
J_{G}(\pi) \;=\;
& -\frac{1}{2}\int_{0}^{T} dt\ \lan \chi(\rho_{t}), (\nabla G_{t})^{2}\ran \\
& -\int_{0}^{T}dt\ \{\lan B(\rho_{t}), e^{G_{t}}-G_{t}-1\ran
+\lan D(\rho_{t}), e^{-G_{t}}+G_{t}-1\ran\} \;.
\end{align*}
Since $e^{x}-x-1\ge0$ for any $x$ in $\R$, $I_{T}(\pi|\ga)=0$.  In
addition, the bound \eqref{bound4.7} follows from Proposition
\ref{energy}.

On the other hand, if $I_{T}(\pi|\ga)$ is equal to $0$, then, for any
$G$ in $C^{1,2}([0,T]\times\T)$ and $\e$ in $\R$, we have $J_{\e
  G}(\pi)\le0$. Note that $J_{0}(\pi)$ is equal to $0$.  Hence the
derivative of $J_{\e G}(\pi)$ in $\e$ at $\e=0$ is equal to $0$. This
implies that the density $\rho$ is a weak solution of the Cauchy
problem \eqref{rdeq}.
\end{proof}

\begin{theorem}
\label{lsc}
The function $I_{T}(\cdot|\ga):D([0,T],\mathcal{M}_{+})\to[0,\infty]$
is lower semicontinuous and has compact level sets.
\end{theorem}

\begin{proof}
For each $q\ge0$, let $E_{q}$ be the level set of
the rate function $I_{T}(\cdot|\ga)$:
\begin{equation*}
E_{q}\;:=\;\{\pi\in D([0,T], \M_{+}) | I_{T}(\pi|\ga)\le q\}\;.
\end{equation*}
Let $\{\pi^{n}:n\ge1\}$ be a sequence in $D([0,T],\M_{+})$ such that
$\pi^{n}$ converges to some element $\pi$ in $D([0,T], \M_{+})$.  We
show that $I_{T}(\pi|\ga)\le\liminf_{n\to\infty} I_{T}(\pi^{n}|\ga)$.
If $\liminf I_{T}(\pi^{n}|\ga)$ is equal to $\infty$, the conclusion
is clear.  Therefore, we may assume that the set
$\{I_{T}(\pi^{n}|\ga): n\ge1\}$ is contained in $E_{q}$ for some
$q>0$.  From the lower semicontinuity of the energy $\mathcal{Q}$ and
Proposition \ref{energy}, we have
\begin{equation*}
\mathcal{Q}(\pi)\;\le\;\varliminf_{n\to\infty}
\mathcal{Q}(\pi^{n}) \;\le\; C(q+1)<\infty\;.
\end{equation*}
Since $\pi^{n}$ belongs to $D([0,T], \M_{+,1})$, so does $\pi$.

Let $\rho$ and $\rho^{n}$ be the density of $\pi$ and $\pi^{n}$
respectively.  We now claim that the sequence $\{\rho^{n} : n\ge1\}$
converges to $\rho$ in $L^{1}([0,T]\times\T)$.  Indeed, by the
triangle inequality,
\begin{equation}
\label{bound4.6}
\begin{split}
& \int_{0}^{T} \  \|\rho_{t} - \rho^{n}_{t}\|_{1}\ dt \\
&\quad  \;\le\; \int_{0}^{T}\ \|\rho_{t} - \rho^{\e}_{t}\|_{1}\ dt +
\int_{0}^{T} \ \|\rho_{t}^{\e} - \rho^{n,\e}_{t}\|_{1}\ dt 
+ \int_{0}^{T} \ \|\rho_{t}^{n,\e} - \rho^{n}_{t}\|_{1}\ dt  \;,  
\end{split}
\end{equation}
where $\rho_{t}^{n,\e}=\rho_{t}^{n}*\psi^{\e}$.
The first term on the right hand side in \eqref{bound4.6} can be computed as
\begin{align*}
\int_{0}^{T}\ \|\rho_{t} - \rho^{\e}_{t}\|_{1}\ dt
& \;\le\; \int_{0}^{T} dt\ \int_{\T} du\ \int_{\T} dv
\ \psi^{\e}(v)|\rho(t,u+v)-\rho(t,u)| \\
& \;\le\; \int_{0}^{T} dt\ \int_{\T} du\ \int_{\T} dv \ \psi^{\e}(v)\int_{u}^{u+v} dw
\ |\nabla \rho(t,w)|\;. \\
\end{align*}
Note that ${\rm supp}\ \psi_{\e} \subset[-\e,\e]$.
From the fundamental inequality $2ab\le A^{-1}a^{2}+Ab^{2}$, for any $A>0$,
the above expression can be bounded above by
\begin{align*}
\dfrac{\mathcal{Q}(\pi)}{2A} +\dfrac{AT\e}{2}\;.
\end{align*}
Similarly, the last term on the right hand side in \eqref{bound4.6}
can be bounded above by
\begin{align*}
\int_{0}^{T}\ \|\rho_{t}^{\e,n} - \rho^{n}_{t}\|_{1}\  dt
\;\le\; \dfrac{\mathcal{Q}(\pi^{n})}{2A} +\dfrac{AT\e}{2}\;.
\end{align*}
Since, for fixed $\e>0$, $\rho^{\e,n}_{t}$ converges to $\rho^{\e}_{t}$ weakly
as $n\to\infty$ for $a.e.\ t\in[0,T]$, 
letting $n\to\infty$ in \eqref{bound4.6} gives that
\begin{equation*}
\varlimsup_{n\to\infty}\int_{0}^{T}\ \|\rho_{t} - \rho^{n}_{t}\|_{1}\ dt \;\le\;
C(q,T)\{\dfrac{1}{A}+ A\e\}\;,
\end{equation*}
for some constant $C(q,T)>0$ which depends on $q$ and $T$.  Optimizing
in $A$ and letting $\e\downarrow0$, we complete the proof of the claim
made above \eqref{bound4.6}.

It follows from this claim that for any function $G$ in
$C^{1,2}([0,T]\times\T)$,
\begin{equation*}
\lim_{n\to\infty}J_{G}(\pi^{n})\;=\;J_{G}(\pi)\;.
\end{equation*}
This limit implies that $I_{T}(\pi|\ga)\le\liminf_{n\to\infty}
I_{T}(\pi^{n}|\ga)$, proving that $I_{T}(\,\cdot\,|\ga)$ is
lower-semicontinuous.

The same argument shows that $E_{q}$ is closed in $D([0,T],\M_{+})$.
Since it is shown in \cite{jlv} that $E_{q}$ is relatively compact in
$D([0,T],\M_{+})$, $E_{q}$ is compact in $D([0,T],\M_{+})$, and the
proof is completed.
\end{proof}

\section{$I_{T}(\cdot|\ga)$-Density}

The lower bound of the large deviations principle stated in Theorem
\ref{mt1} has been established in \cite{jlv} for smooth trajectories.
To remove this restriction, we have to show that any trajectory
$\pi_{t}$, $0\le t \le T$, with finite rate function can be
approximated by a sequence of smooth trajectories $\{\pi^{n}: n\ge1\}$
such that
\begin{align*}
\pi^{n}\longrightarrow\pi
\;\; \text{ and }\;\; 
I_{T}(\pi^{n}|\ga)\longrightarrow I_{T}(\pi|\ga)\;.
\end{align*}
This is the content of this section.  We first introduce some
terminology.

\begin{definition}
Let $A$ be a subset of $D([0,T],\M_{+})$.  $A$ is said to be
$I_{T}(\cdot|\ga)$-dense if for any $\pi$ in $D([0,T],\M_{+})$ such
that $I_{T}(\pi|\ga)<\infty$, there exists a sequence
$\{\pi^{n}:n\ge1\}$ in $A$ such that $\pi^{n}$ converges to $\pi$ in
$D([0,T],\M_{+})$ and $I_{T}(\pi^{n}|\ga)$ converges to
$I_{T}(\pi|\ga)$.
\end{definition}

Let $\Pi$ be the set of all trajectories $\pi(t,du)=\rho(t,u)du$ in
$D([0,T],\M_{+,1})$ whose density $\rho$ is a weak solution of the
Cauchy problem
\begin{equation}
\label{peq}
\begin{cases}
\partial_{t}\rho \;=\; \dfrac{1}{2}\De\rho
-\nabla(\chi(\rho)\nabla H) + B(\rho)e^{H} - 
D(\rho) e^{-H} \ \text{ on }\ \T\;, \\
\rho(0,\cdot) \;=\; \ga(\cdot)\;,
\end{cases}
\end{equation}
for some function $H$ in $C^{1,2}([0,T]\times\T)$.

\begin{theorem}
\label{dense}
Assume that the functions $B$ and $D$ are concave. Then, the set $\Pi$
is $I_{T}(\cdot|\ga)$-dense.
\end{theorem}

The proof of Theorem \ref{dense} is divided into several steps.
Throughout this section, denote by $\lambda:[0,T]\times\T\to[0,1]$ the
unique weak solution of the Cauchy problem \eqref{rdeq} with initial
profile $\ga$, and assume that the functions $B$ and $D$ are concave.

Let $\Pi_{1}$ be the set of all paths $\pi(t,du)=\rho(t,u)du$ in
$D([0,T],\mathcal{M}_{+,1})$ whose density $\rho$ is a weak solution of
the Cauchy problem \eqref{rdeq} in some time interval $[0,\de]$,
$\de>0$.

\begin{lemma}
\label{pi1}
The set $\Pi_{1}$ is $I_{T}(\cdot|\ga)$-dense.
\end{lemma}

\begin{proof}
Fix $\pi(t,du)=\rho(t,u)du$ in $D([0,T],\mathcal{M}_{+,1})$ such that
$I_{T}(\pi|\ga)<\infty$.  For each $\de>0$, set the path
$\pi^{\de}(t,du)=\rho^{\de}(t,u)du$ where
\begin{align*}
\rho^{\de}(t,u) \;=\; 
\begin{cases}
\lambda(t,u) & \text{ if } t\in[0,\de]\;, \\
\lambda(2\de-t,u) & \text{ if } t\in[\de,2\de]\;, \\
\rho(t-2\de,u) & \text{ if } t\in[2\de,T]\;. 
\end{cases}
\end{align*}
It is clear that $\pi^{\de}$ converges to $\pi$ in $D([0,T],\M_{+})$
as $\de \downarrow 0$ and that $\pi^{\de}$ belongs to $\Pi_{1}$.  To
conclude the proof it is enough to show that $I_{T}(\pi^{\de}|\ga)$
converges to $I_{T}(\pi|\ga)$ as $\de \downarrow 0$.

Since the rate function is lower semicontinuous,
$I_{T}(\pi|\ga)\le\liminf_{\de\to0}I_{T}(\pi^{\de}|\ga)$.  Note that
$\mathcal{Q}(\pi^{\de})\le 2\mathcal{Q} (\lambda)+\mathcal{Q}(\pi)$.
From Corollary \ref{zero}, we have $\mathcal{Q}(\pi^{\de})<\infty$.
To prove the upper bound $\limsup_{\de\to0}I_{T}(\pi^{\de}|\ga)\le
I_{T}(\pi|\ga)$, we now decompose the rate function
$I_{T}(\pi^{\de}|\ga)$ into the sum of the contributions on each time
interval $[0,\de]$, $[\de, 2\de]$ and $[2\de, T]$.  The first
contribution is equal to $0$ since the density $\rho^{\de}$ is a weak
solution of the equation \eqref{rdeq} on this interval.  The third
contribution is bounded above by $I_{T}(\pi|\ga)$ since $\pi^{\de}$ on
this interval is a time translation of the path $\pi$.

On the time interval $[\de, 2\de]$, the density $\rho^{\de}$ solves
the backward reaction-diffusion equation:
$\partial_{t}\rho^{\de}=-(1/2)\De\rho^{\de}-F(\rho^{\de})$.
Therefore, the second contribution can be written as
\begin{equation*}
\begin{split}
\sup_{G\in C^{1,2}([0,T]\times\T)}
\Big\{&\int_0^{\de} dt\ \big\{\lan\nabla\lambda_{t},\nabla G_{t}\ran
-\frac{1}{2}\lan \chi(\lambda_{t}), (\nabla G_{t})^{2}\ran \big\} \\
+& \int_{0}^{\de}dt\ \big\{\lan B(\lam_{t}), 1-e^{G_{t}}-G_{t}\ran+
\lan D(\lam_{t}), 1-e^{-G_{t}}+G_{t}\ran \big\} \Big\}\;.
\end{split}
\end{equation*}
By Schwarz inequality, the first integral inside the supremum
is bounded above by
\begin{equation}
\label{exp5.1}
\dfrac{1}{2}
\int_{0}^{\de} dt\ \int_{\T} du\ \frac{|\nabla\lam(t,u)|^{2}}{\chi(\lam(t,u))}\;.
\end{equation}
On the other hand, taking advantage of the relation \eqref{02} and of
the fact that $B$ and $D$ are bounded functions, a simple computation
shows that the second integral inside the supremum in the penultimate
displayed equation is bounded above by
\begin{equation*}
C\int_{0}^{\de} dt\ \int_{\T} du\ \log \frac 1 {\chi(\lam(t,u))}  + C\de\;,
\end{equation*}
for some finite constant $C$ independent of $\de$. By Corollary
\ref{zero}, the expression \eqref{exp5.1} converges to $0$ as
$\de\downarrow0$. Hence, to conclude the proof it suffices to show
that
\begin{equation}
\label{lim1}
\lim_{\de\downarrow0}\int_{0}^{\de} dt\ \int_{\T} du\ \log{\chi(\lam(t,u))}\;=\;0\;.
\end{equation}

Let $\lam_{t}^{j}:[0,T]\to\R$, $j=0,1$, be the unique
solution of the ordinary differential equation
\begin{equation}\label{ode3}
\frac{d}{dt}\lam_{t}^{j} \;=\; F(\lam_{t}^{j})\;,
\end{equation}
with initial condition $\lam_{0}^{j} = j$ and set
$\lambda^j(t,u) \equiv \la_t^j$ for $(t,u)\in[0,T]\times\T$.
Since $\la^j$ is constant in spatial variable, $\De\la^j =0$.
Therefore it follows from \eqref{ode3} that
$\la^j$ is a unique weak solution of the Cauchy problem
\eqref{rdeq} with initial profile $\lam_{0}^{j}(u)\equiv j$.
By Proposition \ref{mono}, 
\begin{equation}
\label{ineq5.1}
\lam^0_{t}\;\le\;\lam(t,u)\ \text{ and }\ 1-\lam^{1}_{t}\;\le\; 1-\lam(t,u)\;,
\end{equation}
for any $(t,u)\in[0,T]\times\T$.
Since $F(1)<0<F(0)$, an elementary computation shows that
\begin{equation}\label{lim5.2}
\lim_{\de\downarrow0}\int_{0}^{\de} dt\ \log{\lam_{t}^{0}} \;=\; 0
\;\; \text{ and }\;\;
\lim_{\de\downarrow0}\int_{0}^{\de} dt\ \log{(1-\lam_{t}^{1})} \;=\; 0\;.
\end{equation}
By definition of  $\chi$ and by \eqref{ineq5.1}, 
\begin{equation*}
\log{\chi(\lam(t,u))} 
\;=\; \log{\lam(t,u)} + \log{(1-\lam(t,u))} 
\ge \log{\lam_{t}^{0}} + \log{(1-\lam_{t}^{1})}\;.
\end{equation*}
To conclude the proof of \eqref{lim1}, it remains to recall
\eqref{lim5.2}.
\end{proof}

Let $\Pi_{2}$ be the set of all paths $\pi(t,du)=\rho(t,u)du $ in
$\Pi_{1}$ with the property that for every $\de>0$ there exists $\e>0$
such that $\e\le\rho(t,u)\le1-\e$ for all $(t,u)\in[\de,T]\times\T$.

\begin{lemma}
\label{pi2}
The set $\Pi_{2}$ is $I_{T}(\cdot|\ga)$-dense.
\end{lemma}

\begin{proof}
Fix $\pi(t,du)=\rho(t,u)du$ in $\Pi_{1}$ such that
$I_{T}(\pi|\ga)<\infty$.  For each $\e>0$, set the path
$\pi^{\e}(t,du)=\rho^{\e}(t,u)du$ with $\rho^{\e}=(1-\e)\rho+\e\lam$.
It is clear that $\pi^{\e}$ converges to $\pi$ in $D([0,T],\M_{+})$ as
$\e \downarrow 0$.  Let $\lambda^{j}(t,u)\equiv\lam_{t}^{j}$, $j=0,1$,
be the weak solution of the equation \eqref{rdeq} with initial profile
$\lam_{0}^{j}(u)\equiv j$. By Proposition \ref{mono}, $\e\lam^{0}
\le\rho^{\e} \le(1-\e)+ \e\lam^{1}$.
Moreover it is easy to see that $\lam^{j}$, $j=1,2$,
belongs to the set $\Pi_{2}$ since $\lam^{j}$ solves
the ordinary differential equation
\begin{equation*}
\frac{d}{dt}\lam_{t}^{j} \;=\; F(\lam_{t}^{j})\;,
\end{equation*}
and $F(1)<0<F(0)$.
Therefore $\pi^{\e}$ belongs to
$\Pi_{2}$.  To conclude the proof it is enough to show that
$I_{T}(\pi^{\e}|\ga)$ converges to $I_{T}(\pi|\ga)$ as $\e \downarrow
0$.

Since the rate function is lower semicontinuous,
$I_{T}(\pi|\ga)\le\liminf_{\e\downarrow0}I_{T}(\pi^{\e}|\ga)$. By the
convexity of the energy, $\mathcal{Q}(\pi^{\e})\le \e\mathcal{Q}
(\lambda)+(1-\e)\mathcal{Q}(\pi)$, hence
$\mathcal{Q}(\pi^{\e})<\infty$.  Let $G$ be a function in
$C^{1,2}([0,T]\times\T)$.  Since $B,D$ and $\chi$ are concave and
Lipschitz continuous,
\begin{align*}
J_{G}(\pi^{\e}) \;\le\;(1-\e)J_{G}(\pi)+\e J_{G}(\lam) 
+ C_0 \{\e+\int_{0}^{T}\ \|\rho_{t}^{\e}-\rho_{t}\|_{1}\ dt\}
\end{align*}
for some finite constant $C_0$, which may change from
line to line. Therefore,
\begin{equation*}
I_{T}(\pi^{\e}|\ga)\;\le\;(1-\e)I_{T}(\pi|\ga) + \e I_{T}(\lambda|\ga)
+ C_0 T \e\;.
\end{equation*}
Letting $\e\downarrow0$ gives $\limsup_{\e\downarrow0}
I_{T}(\pi^{\e}|\ga) \le I_{T}(\pi|\ga)$, which completes the proof.
\end{proof}

Let $\Pi_{3}$ be the set of all paths $\pi(t,du)=\rho(t,u)du$ in
$\Pi_{2}$ whose density $\rho(t,\cdot)$ belongs to the space
$C^{\infty}(\T)$ for any $t\in(0,T]$.

\begin{lemma}
\label{pi3}
The set $\Pi_{3}$ is $I_{T}(\cdot|\ga)$-dense.
\end{lemma}

\begin{proof}
Fix $\pi(t,du)=\rho(t,u)du$ in $\Pi_{2}$ such that
$I_{T}(\pi|\ga)<\infty$.  Since $\pi$ belongs to the set $\Pi_{1}$, we
may assume that the density solves the equation \eqref{rdeq} in some
time interval $[0,2\de]$, $\de>0$.  Take a smooth nondecreasing
function $\a:[0,T]\to[0,1]$ with the following properties:
\begin{align*}
\begin{cases}
\a(t)\;=\;0\    &\text{ if }\ t\in[0, \de]\;, \\
0\;<\a(t)\;<1\  &\text{ if }\ t\in (\de,2\de)\;, \\
\a(t)\;=\;1\    &\text{ if }\ t\in[2\de, T]\;.
\end{cases}
\end{align*}
Let $\psi(t,u):(0,\infty)\times\T\to(0,\infty)$ be the transition
probability density of the Brownian motion on $\T$ at time $t$
starting from $0$.  For each $n\in\N$, denote by $\psi^{n}$ the
function
\begin{equation*}
\psi^{n}(t,u) \;:=\; \psi(\frac{1}{n}\a(t) ,u)
\end{equation*}
and define the path $\pi^{n}(t,du)=\rho^{n}(t,u)du$ where
\begin{align*}
\rho^{n}(t,u) \;=\;
\begin{cases}
\rho(t,u)\  &\text{ if }\ t\in[0, \de]\;, \\
(\rho_{t}*\psi^{n}_{t})(u) \;=\; \int_{\T} dv\ \rho(t,v)\psi^{n}(t,u-v)
\  &\text{ if }\ t\in(\de, T]\;.
\end{cases}
\end{align*}
It is clear that $\pi^{n}$ converges to $\pi$ in $D([0,T],\M_{+})$ as
$n\to\infty$.  Since the density $\rho^{n}$ is a weak solution to the
Cauchy problem \eqref{rdeq} in time interval $[0,\de]$, by Proposition
\ref{reg}, $\rho^{n}(t,\cdot)$ belongs to the space $C^{\infty}(\T)$
for $t\in(0,\de]$.  On the other hand, by the definition of
$\rho^{n}$, it is clear that $\rho^{n}(t,\cdot)$ belongs to the space
$C^{\infty}(\T)$ for $t\in (\de,T]$.  Therefore $\pi^{n}$ belongs to
$\Pi_{3}$.  To conclude the proof it is enough to show that
$I_{T}(\pi^{n}|\ga)$ converges to $I_{T}(\pi|\ga)$ as $n\to\infty$.

Since the rate function is lower semicontinuous,
$I_{T}(\pi|\ga)\le\liminf_{n\to\infty}I_{T}(\pi^{\e}|\ga)$.  Note that
the generalized derivative of $\rho^{n}$ is given by
\begin{align*}
\nabla\rho^{n}(t,u) \;=\; 
\begin{cases}
\nabla\rho(t,u)\  &\text{ if }\ t\in[0, \de]\;, \\
(\nabla\rho_{t}*\psi^{n}_{t})(u)
\  &\text{ if }\ t\in(\de, T]\;.
\end{cases}
\end{align*}
Therefore, by Schwarz inequality, $\mathcal{Q}(\pi^{n}) \le\mathcal{Q}
(\pi)<\infty$.  

The strategy of the proof of the upper bound is similar to the one of
Lemma \ref{pi1}.  We decompose the rate function $I_{T}(\pi^{n}|\ga)$
into the sum of the contributions on each time interval $[0,\de]$,
$[\de, 2\de]$ and $[2\de, T]$.  The first contribution is equal to $0$
since the density $\rho^{n}$ is a weak solution of the Cauchy problem
\eqref{rdeq} on this interval. Since $\pi^{n}$ is defined as a spatial
average of $\pi$, and since the functions $B$ and $D$ are concave,
similar arguments to the ones presented in the proof of Lemma
\ref{pi2} yield that the third contribution is bounded above by
$I_{T}(\pi|\ga)+o_{n}(1)$.  Hence it suffices to show that the second
contribution converges to $0$ as $n\to\infty$.

Since $\partial_{t}\psi = (1/2) \De\psi$, an integration by parts
yields that in the time interval $(\de, 2\de)$,
\begin{equation*}
\partial_{t}\rho^{n} \;=\; \partial_{t}\rho*\psi^{n} +
\frac{\a'(t)}{2n}\De\rho*\psi^{n}\;.
\end{equation*}
Thus, since in the time interval $[\delta, 2\delta]$ $\rho$ is a weak
solution of the hydrodynamic equation \eqref{rdeq}, for any function
$G$ in $C^{1,2}([0,T]\times\T)$,
\begin{align*}
\lan\rho_{2\de}^{n}, G_{2\de}\ran & - \lan\rho_{\de}^{n}, G_{\de}\ran
- \int_{\de}^{2\de} dt\ \lan\rho_{t}^{n}, \partial_{t}G_{t}\ran \\
& =\; \int_{\de}^{2\de} dt\ \Big \{\lan\rho_{t}^{n}, \frac{1}{2}\De G_{t}\ran
-\frac{\a'(t)}{2n}\lan\nabla\rho_{t}^{n}, \nabla G_{t}\ran
+\lan F_{t}^{n}, G_{t}\ran \Big\}\;,
\end{align*}
where $F^{n}_{t} = F(\rho_{t})*\psi^{n}_{t}$.  Therefore, the 
contribution to $I_{T}(\pi|\ga)$ of the piece of the trajectory in the
time interval $[\delta,2\delta]$ can be written as
\begin{equation}
\label{sup5.2}
\begin{split}
\sup_{G\in C^{1,2}([0,T]\times\T)}
\Big\{&\int_{\de}^{2\de} dt\ \Big( -\frac{\a'(t)}{2n}
\lan\nabla\rho_{t}^{n},\nabla G_{t}\ran
-\frac{1}{2}\lan \chi(\rho^{n}_{t}), (\nabla G_{t})^{2}\ran \Big) \\
+& \int_{\de}^{2\de}dt\ \lan F_{t}^{n}G_{t} - B(\rho_{t}^{n})(e^{G_{t}}-1)
-D(\rho_{t}^{n})(e^{-G_{t}}-1)\ran \Big\}\;.   
\end{split}
\end{equation}

By Schwarz inequality, the first integral inside the supremum
is bounded above by
\begin{equation*}
\frac{\|\a'\|^{2}_{\infty}}{8n^{2}}\int_{\de}^{2\de} dt\ 
\int_{\T}du\ \frac{|\nabla\rho^{n}(t,u)|^{2}}{\chi(\rho^{n}(t,u))}\;.
\end{equation*}
Since $\pi$ belongs to $\Pi_{2}$, there exists a positive constant
$C(\de)$, depending only on $\de$, such that
$C(\de)\le\rho^{n}\le1-C(\de)$ on time interval $[\de, 2\de]$.  This
bounds together with the fact that
$\mathcal{Q}(\pi^{n})\le\mathcal{Q}(\pi)$ permit to prove that the
previous expression converges to $0$ as $n\to\infty$.  On the other
hand, the second integral inside the supremum \eqref{sup5.2} is
bounded above by
\begin{equation}
\label{int5.1}
\int_{\de}^{2\de}dt\ \lan F_{t}^{n}m_{t}^{n} 
- B(\rho_{t}^{n})(e^{m_{t}^{n}}-1)
-D(\rho_{t}^{n})(e^{-m_{t}^{n}}-1)\ran\;,
\end{equation}
where
\begin{equation*}
m_{t}^{n}\;=\;\log{\frac{F_{t}^{n}+\sqrt{(F_{t}^{n})^{2}
+4B(\rho_{t}^{n})D(\rho_{t}^{n})}} {2B(\rho_{t}^{n})}}\;.
\end{equation*}
Note that $m_{t}^{n}$ is well-defined and that the integrand in
\eqref{int5.1} is uniformly bounded in $n$ because in the time
interval $[\de, 2\de]$ $\rho_{t}$ is bounded below by a strictly
positive constant and bounded above by a constant strictly smaller
than $1$.  Since $m^{n}(t,u)$ converges to $0$ as $n\to\infty$ for any
$(t,u)\in[\de,2\de]\times\T$, the expression in \eqref{int5.1}
converges to $0$ as $n\to\infty$.
\end{proof}

Let $\Pi_{4}$ be the set of all paths $\pi(t,du)=\rho(t,u)du$ in
$\Pi_{3}$ whose density $\rho$ belongs to $C^{\infty,\infty}(
(0,T]\times \bb T)$.

\begin{lemma}
\label{pi4}
The set $\Pi_{4}$ is $I_{T}(\cdot|\ga)$-dense.
\end{lemma}

\begin{proof}
Fix $\pi(t,du)=\rho(t,u)du$ in $\Pi_{3}$ such that
$I_{T}(\pi|\ga)<\infty$.  Since $\pi$ belongs to the set $\Pi_{1}$, we
may assume that the density $\rho$ solves the equation \eqref{rdeq} in
the time interval $[0,3\de]$ for some $\de>0$.  Take a smooth
nonnegative function $\phi:\R\to\R$ with the following properties:
\begin{align*}
\text{supp}\ \phi\;\subset\;[0,1] \;\; \text{ and }\;\; 
\int_{0}^{1}\phi(s)ds \;=\; 1\;.
\end{align*}

Let $\a$ be the function introduced in the previous lemma.  For each
$\e>0$ and $n\in\N$, let
\begin{equation*}
\Phi(\e,s)\;:=\; \frac{1}{\e}\phi(\frac{s}{\e})\;, \quad 
\a_{n}(t) \;:=\; \frac{1}{n}\a(t)\;,
\end{equation*}
and let $\pi^{n}(t,du)=\rho^{n}(t,u)du$ where
\begin{align*}
\rho^{n}(t,u) \;=\; \int_{0}^{1}\rho(t+\a_{n}(t)s,u)\phi(s) ds
\;=\; \int_{\R} \rho(t+s,u)\Phi(\a_{n}(t), s) ds\;.
\end{align*}
In the above formula, we extend the definition of $\rho$ to $[0,T+1]$
by setting $\rho_{t}=\tilde{\lambda}_{t-T}$ for $T \le t \le T+1$,
where $\tilde{\lambda} :[0,1]\times\T\to[0,1]$ stands for the unique
weak solution of the equation \eqref{rdeq} with initial profile
$\rho_{T}$. Note that it follows from the construction of $\rho^n$ that
$\rho_t=\rho_t^n$ for any $t\in[0,\de]$, therefore, $\rho^n$ is
a weak solution of the equation \eqref{rdeq} in time interval $[0,\de]$.

It is clear that $\pi^{n}$ converges to $\pi$ in
$D([0,T],\mathcal{M}_{+})$.  Since on the time interval $(0, 3\de)$,
the function $\rho$ is smooth in time, for $n$ large enough the
function $\rho^{n}$ is smooth in time on $(0,T]\times \bb T$.
Hence, $\pi^{n}$ belongs to $\Pi_{4}$ and $\mathcal{Q}(\pi^{n})$ is finite.

The remaining part of the proof is similar to the one of the previous
lemma. We only present the arguments leading to the bound
$\limsup_{n\to\infty} I_{T}(\pi^{n}|\ga) \le I_{T}(\pi|\ga)$. The rate
function can be decomposed in three pieces, two of which can be
estimated as in Lemma \ref{pi3}. We consider the contribution to
$I_{T}(\pi^{n}|\ga)$ of the piece of the trajectory corresponding to
the time interval $[\de, 2\de]$.

The derivative of $\rho^{n}$ in time on $(\de, 2\de)$ is computed as
\begin{equation*}
\partial_{t}\rho^{n}(t,u) \;=\;
\int_{\R} \partial_{t}\rho(t+s,u)\Phi(\a_{n}(t),s) ds
\;+\; \int_{\R}\rho(t+s,u) \partial_{t} [\Phi(\a_{n}(t), s)] ds\;.
\end{equation*}
It follows from this equation and from the fact that the density
$\rho$ solves the hydrodynamic equation \eqref{rdeq} on the time
interval $[\de,3\de]$, that for any function $G$ in
$C^{1,2}([0,T]\times\T)$,
\begin{equation*}
\lan\rho_{2\de}^{n}, G_{2\de}\ran - \lan\rho_{\de}^{n}, G_{\de}\ran
- \int_{\de}^{2\de} dt\ \lan\rho_{t}^{n}, \partial_{t}G_{t}\ran 
\;=\; \int_{\de}^{2\de} dt\ \Big\{\lan\rho_{t}^{n}, \frac{1}{2}\De G_{t}\ran
+\lan F_{t}^{n}+r_{t}^{n}, G_{t}\ran \Big\}\;,
\end{equation*}
where
\begin{align*}
& F^{n}(t,u) \;:=\; \int_{\R}F(\rho(t+s,u)) \, \Phi(\a_{n}(t),s) \, ds\;, \\
&\quad r^{n}(t,u) \;:=\; \int_{\R}\rho(t+s,u) \,
\partial_{t} [\Phi(\a_{n}(t),s)] \, ds\; .
\end{align*}
Therefore, the second contribution can be bounded above by
\begin{align}
\label{sup5.3}
\sup_{G\in C^{1,2}([0,T]\times\T)}
\Big\{ \int_{\de}^{2\de}dt\ \lan (F_{t}^{n}+r_{t}^{n})G_{t} 
- B(\rho_{t}^{n})(e^{G_{t}}-1)
-D(\rho_{t}^{n})(e^{-G_{t}}-1)\ran \Big\}\;.
\end{align}

We now show that $r^{n}(t,u)$ converges to $0$ as $n\to\infty$
uniformly in $(t,u)\in(\de,2\de)\times\T$.  Let $(t,u)$ in
$(\de,2\de)\times\T$.  Since $\int_{\R}\partial_{t}[\Phi(\a_{n}(t),
s)] ds = \partial_{t}[\int_{\R}\Phi(\a_{n}(t), s) ds ] = 0$,
$r^{n}(t,u)$ can be written as
\begin{equation*}
\int_{\R}\{\rho(t+s,u)-\rho(t,u)\} \, \partial_{t}[\Phi(\a_{n}(t),
s)]\, ds\; .
\end{equation*}
Since $\rho$ is Lipschitz continuous on $[\de, 3\de]\times\T$, there
exists a positive constant $C(\de)>0$, depending only on $\de$, such
that
\begin{equation*}
|\rho(t+s,u)-\rho(t,u)| \;\le\; C(\de)s \;,
\end{equation*}
for any $(t,u)\in[\de, 2\de]\times\T$ and $s\in[0,\de]$.  Therefore
$r^{n}(t,u)$ is bounded above by
\begin{equation*}
C(\de)\int_{\R} s \, \big|\partial_{t} [\Phi(\a_{n}(t), s)]\big| \, ds \;.
\end{equation*}
It follows from a simple computation and from the change of variables
$\a_{n}(t)s=\bar{s}$ that
\begin{equation*}
\int_{\R} s \, \big|\partial_{t} [\Phi(\a_{n}(t), s)] \big| \, ds
\;\le\; \frac{\|\a'(t)\|_{\infty}}{n} \int_{0}^{1} \big\{ 
s\phi(s) + s^{2}|\phi'(s)| \big\}\, ds\; .
\end{equation*}
Therefore $r^{n}(t,u)$ converges to $0$ as $n\to\infty$ uniformly
in $(t,u)\in(\de,2\de)\times\T$.

To complete the proof, it remains to take a supremum in $G\in
C^{1,2}([0,T]\times\T)$ in formula \eqref{sup5.3} and to let
$n\to\infty$.
\end{proof}

\begin{proof}[Proof of Theorem \ref{dense}]
From the previous lemma, all we need is to prove that $\Pi_{4}$ is
contained in $\Pi$.  Let $\pi(t,du)=\rho(t,u)du$ be a path in
$\Pi_{4}$.  There exists some $\de>0$ such that the density $\rho$
solves the equation \eqref{rdeq} on time interval $[0,2\de]$.  In
particular, the density $\rho$ also solves the equation \eqref{peq}
with $H=0$ on time interval $[0,2\de]$.  On the one hand, since the
density $\rho$ is smooth on $[\de,T]$ and there exists $\e>0$ such
that $\e\le\rho(t,u)\le1-\e$ for any $(t,u)\in[\de,T]\times\T$, from
Lemma 2.1 in \cite{jlv}, there exits a unique function $H$ in
$C^{1,2}([\de,T]\times\T)$ satisfying the equation \eqref{peq} with
$\rho$ on $[\de,T]$, and it is proved that $\pi$ belongs to $\Pi$.
\end{proof}

\begin{proof}[Proof of Theorem \ref{mt1}]
We have already proved in Section \ref{sec4} that the rate function is
lower semicontinuous and that it has compact level sets.

Recall from the beginning of this section the definition of the set
$\Pi$. It has been proven in \cite{jlv} that for each closed subset
$\mathcal{C}$ of $D([0,T],\mathcal{M}_{+})$,
\begin{equation*}
\varlimsup_{N\to\infty}\frac{1}{N}\log{Q_{\eta^{N}}(\mathcal{C})} \;\le\; 
-\inf_{\pi\in\mathcal{C}}I_{T}(\pi|\ga)\;,
\end{equation*}
and that for each open subset $\mathcal{O}$ of
$D([0,T],\mathcal{M}_{+})$,
\begin{equation*}
\varliminf_{N\to\infty}\frac{1}{N}\log{Q_{\eta^{N}}(\mathcal{O})} \;\ge\; 
-\inf_{\pi\in\mathcal{O}\cap\Pi}I_{T}(\pi|\ga)\;.
\end{equation*}
Since $\mathcal{O}$ is open in $D([0,T],\mathcal{M}_{+})$, by
Theorem \ref{dense},
\begin{equation*}
\inf_{\pi\in\mathcal{O}\cap\Pi}I_{T}(\pi|\ga)
\;=\;\inf_{\pi\in\mathcal{O}}I_{T}(\pi|\ga)\;,
\end{equation*}
which completes the proof.
\end{proof}

\section{Appendix}

In sake of completeness, we present in this section results on the
Cauchy problem \eqref{rdeq}.  

\begin{definition}
\label{wsol}
A measurable function $\rho:[0,T]\times\T\to[0,1]$ is said to be a
weak solution of the Cauchy problem \eqref{rdeq} in the layer
$[0,T]\times\T$ if, for every function $G$ in
$C^{1,2}([0,T]\times\T)$, 
\begin{align}
\label{weq}
\lan \rho_{T}, G_{T}\ran - \lan \gamma, G_{0} \ran
&- \int_{0}^{T} dt \lan \rho_{t}, \partial_{t}G_{t} \ran  \notag \\
&=\;\dfrac{1}{2} \int_{0}^{T} dt \lan \rho_{t}, \De G_{t} \ran 
+ \int_{0}^{T} dt \lan F(\rho_{t}), G_{t} \ran\;.
\end{align} 
\end{definition}
For each $t\ge0$, let $P_{t}$ be the semigroup on $L^{2}(\T)$ generated by
$(1/2)\De$.

\begin{definition}\label{msol}
A measurable function $\rho:[0,T]\times\T\to[0,1]$
is said to be a mild solution of the Cauchy problem \eqref{rdeq}
in the layer $[0,T]\times\T$ if, for any t in $[0,T]$,
it holds that
\begin{align}\label{meq}
\rho_{t} \;=\; P_{t}\gamma + \int_{0}^{t} P_{t-s}F(\rho_{s}) ds\;.
\end{align} 
\end{definition}

The first proposition asserts existence and uniqueness of weak and
mild solutions, a well known result in the theory of partial
differential equations.  We give a brief proof because uniqueness of
the Cauchy problem \eqref{rdeq} plays an important role in the proof
of Theorem \ref{hydrodynamics}.

\begin{proposition}
\label{exun}
Definitions \ref{wsol} and \ref{msol} are equivalent. Moreover, there
exists a unique weak solution of the Cauchy problem \eqref{rdeq}.
\end{proposition}

\begin{proof}
Since $F$ is Lipschitz continuous, 
by the method of successive approximation,
there exists a unique mild solution of the Cauchy problem \eqref{rdeq}.
Therefore to conclude the proposition it is enough to show that
the above two notions of solutions are equivalent.

Assume that $\rho:[0,T]\times\T\to[0,1]$ is a weak solution of the
Cauchy problem \eqref{rdeq}.  Fix a function $g$ in $C^{2}(\T)$ and
$0\le t \le T$.  For each $\de>0$, define the function $G^{\de}$ as
\begin{align*}
G^{\de}(s,u) \;=\; 
\begin{cases}
     (P_{t-s}g)(u)  & \text{if $0\le s\le t$}\;, \\
     \de^{-1}(t+\de-s)g(u)  & \text{if $t\le s \le t+\de$}\;,  \\
     0  & \text{if $t+\de\le s \le T$}\;.
\end{cases}
\end{align*}
One can approximate $G^{\de}$ by functions in
$C^{1,2}([0,T]\times\T)$.  Therefore, by letting $\de\downarrow0$ in
\eqref{weq} with $G$ replaced by $G^{\de}$ and by a summation by
parts, 
\begin{align}
\label{meq1}
\lan \rho_{t}, g\ran \;=\; \lan P_{t}\gamma,g \ran
 + \int_{0}^{t} \lan P_{t-s}F(\rho_{s}),g \ran ds\;.
\end{align}
Since \eqref{meq1} holds for any function $g$ in $C^{2}(\T)$,
$\rho$ is a mild solution of the Cauchy problem \eqref{rdeq}.

Conversely, assume that $\rho:[0,T]\times\T\to[0,1]$ is a weak
solution of the Cauchy problem \eqref{rdeq}.  In this case,
\eqref{meq1} is true for any function $g$ in $C^{2}(\T)$ and any $0\le
t\le T$.  Differentiating \eqref{meq1} in $t$ gives that
\begin{equation*}
\dfrac{d}{dt} \lan \rho_{t}, g\ran \;=\;
\dfrac{1}{2}\lan \rho_{t}, \De g \ran + \lan F(\rho_{t}), g \ran\;.
\end{equation*}
Therefore \eqref{weq} holds for any function $G(t,u)=g(u)$ in $C^{2}(\T)$.
It is not difficult to extend this to any function $G$ in
$C^{1,2}([0,T]\times\T)$.
Hence $\rho$ is a weak solution of the Cauchy problem \eqref{rdeq}.
\end{proof}

The following two propositions assert the smoothness and the monotonicity of 
weak solutions of the Cauchy problem \eqref{rdeq}.

\begin{proposition}
\label{reg}
Let $\rho$ be the unique weak solution of the Cauchy problem \eqref{rdeq}.
Then $\rho$ is infinitely differentiable over $(0,\infty)\times\T$.
\end{proposition}

\begin{proposition}
\label{mono}
Let $\rho_{0}^{1}$ and $\rho_{0}^{2}$ be two initial profiles.
Let $\rho^{j}$, $j=1,2$, be the weak solutions of
the Cauchy problem \eqref{rdeq} with initial condition
$\rho_{0}^{j}$.
Assume that
\begin{equation*}
m\{u\in\T:\rho_{0}^{1}(u)\le\rho_{0}^{2}(u)\} \;=\; 1\;,
\end{equation*}
where $m$ is the Lebesgue measure on $\T$.
Then, for any $t\ge0$, it holds that
\begin{equation*}
m\{u\in\T\ :\ \rho^{1}(t,u)\le\rho^{2}(t,u)\} \;=\; 1\;.
\end{equation*}
\end{proposition}
The proofs of Propositions \ref{reg} and \ref{mono}
can be found in the ones of Proposition 2.1 of \cite{ds}.

The last proposition asserts that, for any initial density profile
$\gamma$, the weak solution $\rho_{t}$ of the Cauchy problem
\eqref{rdeq} converges to some solution of the semilinear elliptic
equation \eqref{seeq}.  Recall, from Subsection \ref{hs}, the
definition of the set $\mathcal{E}$.  

\begin{proposition}
\label{conv}
Let $\rho:[0,\infty)\times\T\to[0,1]$ be the unique weak solution of
the Cauchy problem \eqref{rdeq}.  Then there exists a density profile
$\rho_{\infty}$ in $\E$ such that $\rho_{t}$ converges to
$\rho_{\infty}$ as $t\to\infty$ in $C^{2}(\T)$.
\end{proposition}

The proof of this proposition can be found in
the one of Theorem D of \cite{cm}.

\section*{Acknowledgements}
The authors wish to thank J. Farf\'an for the careful reading
of the first version of the manuscript and his fruitful comments
and an anonymous referee for his or her helpful comments.
This research work has been initiated during the second author's stay
in Instituto Nacional de Matem\'atica Pura e Aplicada,
he acknowledges gratefully hospitality and support.
He would like to thank Professor Funaki for his constant encouragement.

\end{document}